\documentclass[11pt,a4paper, final]{article}
\usepackage{template}
\usepackage{calc}
\usepackage{pgfplots}
\usepackage{fullpage}

\usepackage[normalem]{ulem}

\newcommand{\Addr}{A}
\DeclareMathOperator\cut{Cut}

\DeclareMathOperator\E{E}
\let\P\relax
\DeclareMathOperator\P{P}

\newcommand{\notleftright}{\mathrel{\ooalign{$\leftrightarrow$\cr\hidewidth$/$\hidewidth}}}

\newcommand{\PP}{\mathbf{P}}

\newcommand{\Z}{\mathbb{Z}}

\newcommand{\R}{\mathbb{R}}
\newcommand{\N}{\mathbb{N}}

\colorlet{roxo}{purple!20!white}
\colorlet{verde}{green!20!white}
\colorlet{azul}{blue!20!white}

\makeatletter
\newcommand\xleftrightarrow[2][]{%
  \ext@arrow 9999{\longleftrightarrowfill@}{#1}{#2}}
\newcommand\longleftrightarrowfill@{%
  \arrowfill@\leftarrow\relbar\rightarrow}
\makeatother

\newcounter{iconst}

\newcommand{\Edge}{\mathcal{E}}

\usetikzlibrary{patterns}

\begin{document}

\title{Percolation on hierarchical lattices}

\date{\today}
\author{
  Caio Alves
  \thanks{Email: \ \texttt{caio\_teodoro.de\_magalhaes\_alves@technikum-wien.at}; \ Department of Computer Science \& Applied Mathematics, FH Technikum Wien, H\"{o}chst\"{a}dtplatz 6, 1200 Vienna - Austria.
}
  \and
  Rangel Baldasso
  \thanks{Email: \ \texttt{rangel@puc-rio.br}; \ Department of Mathematics, PUC-Rio, Rua Marqu\^es de S\~ao Vicente 225, G\'avea, 22451-900 Rio de Janeiro, RJ - Brazil.}
  \and
  Carlos Gustavo Moreira
  \thanks{Email: \ \texttt{gugu@impa.br}; \ IMPA, Estrada Dona Castorina 110, 22460-320 Rio de Janeiro, RJ - Brazil.}
  \and
  Augusto Teixeira
  \thanks{Email: \ \texttt{augusto@impa.br}; \ IMPA, Estrada Dona Castorina 110, 22460-320 Rio de Janeiro, RJ - Brazil.}
}

\maketitle

\begin{abstract}
  We consider independent Bernoulli percolation on top of sequences of hierarchical graphs.
  Given a graph $G_{1}$ with two distinguished vertices $a_{1}$ and $b_{1}$, the hierarchical graph with seed $G_{1}$ is the sequence $\big( G_{k} \big)_{k \geq 1}$ resulting from the inductive procedure, where the graph $G_{k+1}$ is obtained from $G_{k}$ by replacing each of its edges with a copy of $G_{1}$, attached by the vertices $a_{1}$ and $b_{1}$.
  We prove that, under sharp hypotheses, percolation on these graphs presents a unique phase transition. Second, we establish the existence of several critical exponents in this context, such as the critical exponents for the correlation length $\nu$, the surface tension $\mu$, the one-arm exponent $\alpha_{1}$.
  Several results are also obtained for their infinite counterpart $G_\infty$, which is the Benjamini-Schramm limit of $G_k$: uniqueness of the infinite cluster, continuity of $\theta(p)$, existence of the percolation-probability exponent $\beta$ and scaling relations for the critical exponents $\alpha_1$, $\nu$ and $\beta$.
  Furthermore, we analyze noise sensitivity for crossing functions in $G_{k}$ and establish sharp noise sensitivity in this setting.
  Finally, we propose a setup where it is possible to verify the locality hypothesis, stating that the critical threshold for percolation is a local property, while critical exponents are determined by the global geometry of the graph.
  As a consequence of the techniques developed here, we also provide a necessary and sufficient condition for the existence of a unique fixed point for the map $p \mapsto \mathbb{E}_p[g]$ in $(0,1)$, where $g:\{0,1\}^n \to \{0,1\}$ is a nontrivial monotone Boolean function.
  \medskip

  \noindent
  \emph{Keywords and phrases.} Hierarchical graphs, percolation, critical exponents, scaling relations.

  \noindent
  MSC 2010: 60G18, 60K35, 82B43, 37D40.
\end{abstract}

\section{Introduction}

  Percolation on Euclidean lattices has been intensively studied in the last decades.
  However, some of the most exciting questions on the subject, such as: the continuity of the phase transition, universality and scaling relations have only been answered for some specific examples, such as the triangular lattice ~\cite{smirnov2001critical, lsw_2002} or higher dimensions~\cite{hara1990mean}.

  Given the big challenge involved in such study, simpler networks have been analyzed as a testbed for the study of the Euclidean case.
  Most notably, percolation on regular trees is an example of success in this type of simplification, since it allows for explicit calculations to be performed, while keeping some of the phenomenology of $\mathbb{Z}^d$~\cite{lyons1990random} and Section 10.1 from~\cite{Gri99}.
  However, several interesting phenomena that can be observed for percolation on $\mathbb{Z}^d$ are not present on trees, such as: noise sensitivity of crossing probabilities~\cite{bks}, non mean-field behavior~\cite{kesten1987scaling, kesten1987scaling2}, and more.

  In this paper, we deepen the study of percolation on recursively defined hierarchical lattices.
  These have been introduced as yet another model that can be analyzed in depth with rigorous mathematics~\cite{levitan, hk2009}.
  However, unlike the example of trees, these models capture with much more fidelity the phenomena that are present on Euclidean lattices.
  In particular, all the examples of shortcomings that were mentioned in the above paragraph can be observed for percolation on hierarchical lattices, as we prove below.

  Another special advantage of hierarchical lattices, is that their study can be understood as a toy model for the Renormalization Group, which in the physics, non-rigorous literature can provide deep insights and conjectures for percolation and other models of statistical mechanics on $\mathbb{Z}^d$.

  The main contributions of this article stem from the depth and precision of the results that we present for percolation on hierarchical lattices.
  In particular, as we state in the subsection below, we prove: uniqueness of the phase transition, existence of a variety of critical exponents, scaling relations, noise sensitivity, locality of critical thresholds, globality of critical exponents and near critical scaling.

\subsection{Notation and main results}

We start by introducing the main object of our study, which consists of a sequence of finite graphs.
For this, let $G_1 = (V_1, \Edge_1)$ be a connected finite graph with two distinguished vertices $a_{1}$ and $b_{1}$. Our central hypothesis throughout the text will be that
\begin{equation}
  \tag{H}
  \label{e:hypothesis}
  \begin{array}{c}
    \arraypar{$G_{1}$ is such that $\d_{G_{1}}(a_1, b_1) > 1$ in $G_1$ and there exist at least two edge-disjoint paths connecting $a_{1}$ to $b_{1}$ in $G_1$.}
  \end{array}
\end{equation}
As explained in Remark~\ref{rmk:cut_set}, the above hypothesis is necessary for the non-triviality of the percolation process that we will introduce.

We now define our sequence of graphs $G_{k}$ by an inductive procedure.
Noting that $G_1$ has already been defined, suppose that we have constructed $G_{k} = (V_{k}, \Edge_{k})$ for some $k \geq 1$.
Then, the graph $G_{k+1}$ is obtained by replacing each edge $e=\{x,y\} \in \Edge_{k}$ by a copy of $G_{1}$, where the distinguished vertices $a_{1}$ and $b_{1}$ are identified with $x$ and $y$, respectively.
See Figure~\ref{f:d_k} for an illustration of the first few graphs in this sequence, for a classic choice of $G_1$.

\begin{remark}
  Note that there may be a choice on how to associate $\{x, y\}$ to $\{a_1, b_{1}\}$ so that the above definition is not always precise.
  In \cref{sec:notation},  we make an oriented construction of $(G_k)_{k \geq 0}$, which will remove this ambiguity.
  However, for the introduction, one may think of graphs $G_{1}$ for which there is an isomorphism that exchanges $a_1$ and $b_1$, see Figure~\ref{f:d_k} for one such example.
  This assumption removes any ambiguity, since both assignments of $\{x, y\}$ with $\{a_1, b_{1}\}$ are equivalent (up to isomorphisms).
\end{remark}

\begin{figure}[h]
  \begin{center}
    \includegraphics[width=0.7\textwidth]{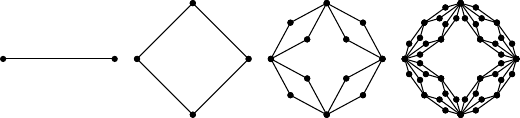}
  \end{center}
  \caption{First three interactions of the graph $G_{k}$ for the diamond hierarchical lattice.}
  \label{f:d_k}
\end{figure}

\medskip

The main goal of this paper is to analyze the behavior of Bernoulli percolation on the sequence of graphs $G_{k}$.
More precisely, given $p \in [0,1]$, denote by $\mathbb{P}^{k}_{p}$ the distribution of independent edge percolation on the graph $G_{k}$ with parameter $p$.
Under this law, every edge is declared open with probability $p$ or closed with probability $1 - p$, independently.

\paragraph{Uniqueness of phase transition.}
The first observable that we focus on is akin to the box-crossing probabilities, that play a central role in the analysis of planar percolation. Denote by $f_{k}:[0,1] \to [0,1]$ the probability that $a_{k}$ is connected to $b_{k}$ by an open path
\begin{equation}
f_k(p) = \mathbb{P}^{k}_p ( a_k \leftrightarrow b_k ).
\end{equation}
The functions above will be used to detect critical behavior for percolation on these graphs.
Theorem~\ref{t:crossing} proves that, under \cref{e:hypothesis},
\begin{display}
  $f_1$ admits a unique fixed point $p_{\star} = p_{\star}(G_{1})$ in the open interval $(0,1)$.
\end{display}
Note that for a given graph like $G_1$ in Figure~\ref{f:d_k}, the above statement can be easily proved by writing down $f_1$ and inspecting it.
However, under the generality of \cref{e:hypothesis}, this result is novel and non-trivial.

Furthermore, \cref{t:crossing} also shows that
\begin{equation}
  \label{e:zeta}
  \zeta = f_{1}'(p_{\star})>1
\end{equation}
and that the crossing probabilities actually determine this critical point, in the sense that
\begin{equation}
p_{\star} \text{ can be characterized by } p_{\star} = \sup \{p \in [0, 1]: \lim_k f_k(p) = 0 \}.
\end{equation}

\paragraph{Correlation length.}
Having proved the existence of a phase transition for percolation on such sequences of graphs, we focus on the three distinct phases of the process. We start with the off-critical phases by establishing the existence of the critical exponents for the correlation length and the surface tension. In fact, Theorem~\ref{th:off-critical-decay} establishes that the limit
\begin{equation}
  \xi(p) = \lim_k \frac{\log f_k(p)}{\d_{G_{1}}(a_1, b_1)^k}
\end{equation}
exists for any $p < p_\star$ and
\begin{equation}\label{eq:critial_correlation}
\uc{c:corr_exp_1} (p_\star - p)^{-\nu}
   \leq \xi(p)
   \leq \uc{c:corr_exp_2} (p_\star - p)^{-\nu},
   \text{ for $p \in (p_\star/2, p_\star)$}
\end{equation}
where $\nu = \log_\zeta (\d_{G_{1}}(a_1, b_1))$, effectively determining the correlation-length critical exponent. Analogous results are obtained for the surface tension in the supercritical phase.

\paragraph{Sharp noise sensitivity.}
In Section~\ref{s:sharp_ns}, we analyze noise sensitivity of the crossing functions in the graphs $G_{k}$. This measures how sensitive percolation crossings are with respect to small noises in the output of the edge-configuration. Denoting by $\omega_{p}$ the percolation coniguration at $G_{k}$ and by $\omega_{p}^{\varepsilon}$ an $\varepsilon$-perturbation of $\omega_{p}$ (see Equation~\eqref{eq:resampling}), we prove in Theorem~\ref{t:sharp_ns} that there exists a non-decreasing homeomorphism $g:[0, +\infty] \to [0,1]$ such that
\begin{equation}
\lim_{k \to \infty} \mathbb{E}_{k} \big[
    \textnormal{\textbf{1}}_{\{a_{k} \leftrightarrow b_{k}\}}(\omega_{p_{\star}})
    \textnormal{\textbf{1}}_{\{a_{k} \leftrightarrow b_{k}\}}(\omega_{p_{\star}}^{\varepsilon_{k}})\big]
    = p_{\star}^{2} g(c)+ p_{\star}(1-g(c)),
\end{equation}
whenever $\zeta^{k}\varepsilon_{k} \to c$. This not only establishes noise sensitivity if $c = +\infty$, but also determines the whole correlation structure observed as the noise level is varied. 

\paragraph{Benjamini-Schramm limit.}
In Section~\ref{sec:bensch_limit} we establish the existence of an infinite (edge-rooted) graph $(\bar{G}_{\infty}, E_{\infty})$ obtained as the Benjamini-Schramm limit of the graphs $G_k$.
In Section~\ref{sec:bensch_perc}, we analyze the percolative properties of this infinite graph, proving in Theorem~\ref{t:p_c_constant} that the critical value for bond percolation $p_{c}(\bar{G}_{\infty})$ is almost surely constant and that
\begin{display}
$p_{c}$ coincides with the box-crossing threshold $p_{\star}$.
\end{display}

In \cref{t:uniqueness} we prove uniqueness of the infinite cluster on $\bar{G}_{\infty}$ for $p > p_{c}$ and in \cref{t:continuity} we conclude that
\begin{display}
$p \mapsto \theta^{\bar{G}_{\infty}}(p) = \PP^{\bar{G}_\infty}_p ( E_\infty \leftrightarrow \infty )$ is almost surely continuous.
\end{display}

\paragraph{Critical cluster.}
In Section~\ref{s:arm_exponents} we start to examine more precisely the structure of percolation clusters in the hierarchical graphs.
It is important to notice that some of the results presented in this section have a similar version presented for the diamond lattice in \cite{hk2009}, but the precise distinctions will be discussed in \cref{s:arm_exponents} itself.
We first analyze the arm exponents and prove in Theorem~\ref{th:volume-of-critical-cluster} that, if $E_k$ denotes an uniformly selected edge in $G_k$, then
\begin{display}
  \label{e:one_arm}
  $\frac{\mathbb{P}_{p_\star}^{k} ( E_k \leftrightarrow \{a_k, b_k\} )}{\d_{G_{k}}(a_k, b_k)^{-\alpha_1}}$ is bounded from above and below by positive constants,
\end{display}
where $\alpha_1$ is the one-arm exponent and can be computed explicitly from the graph $G_{1}$.

After this is established, we examine the size of the connected components of the vertices $a_{k}$ and $b_{k}$ at criticality. That is, if $\mathcal{C}_{\{a_{k}, b_{k}\}}$ denotes the collection of edges in $G_{k}$ such that at least one endvertex is connected by an open path to the set $\{a_{k}, b_{k}\}$, we prove in Theorem~\ref{t:volume_fluctuations} that there exists $d_{f} >0$ such that (at $p_\star$)
\begin{display}
  \label{eq:critical_cluster}
  $\frac{|\mathcal{C}_{\{a_{k}, b_{k}\}}|}{d_{f}^{k}} \underset{k}{\Longrightarrow} \mu_{\infty}$, where $\mu_{\infty}$ is a probability measure supported on $(0, \infty)$.
\end{display}
The constant $d_{f}$ is called the dimension of the critical cluster and bounds for it can be found in Theorem~\ref{th:volume-of-critical-cluster}.

\paragraph{Scaling relations.}
In \cref{th:beta_and_relations} we prove that there exists $\uc{c:beta_and_relations} > 0$ such that
\begin{equation}
  \label{e:bound_theta_th_intro}
  \uc{c:beta_and_relations}^{-1} (p - p_c)^\beta
  \leq
  \textnormal{E} \big[ \theta^{\bar{G}_{\infty}} (p) \big]
  \leq
  \uc{c:beta_and_relations} (p - p_c)^\beta,
\end{equation}
where
\begin{equation}
  \label{eq:exponent_relation_intro}
  \beta = \log_\zeta \frac{|\Edge_{1}|}{d_f}
  = \alpha_1 \nu,
\end{equation}
see the definitions of $\zeta, \alpha_1$, $\nu$ and $d_f$ in \cref{e:zeta}, \cref{e:one_arm}, \cref{eq:critial_correlation}, and \cref{eq:critical_cluster} above.

\begin{remark}
  \label{r:exponent_relation_lattice}
  The relation in \cref{eq:exponent_relation_intro} is the exact analogue of a known scaling relation in Bernoulli percolation on the lattice, see for instance (7.25) in \cite{Nolin_near-critical}.
\end{remark}

\paragraph{Near-critical scaling.}
In \cref{th:near-critical-scaling} of \cref{s:near-critical} we describe in fine detail the near-critical behavior of the crossing probability.
More precisely, we prove that, for any $\lambda \in \mathbb{R}$,
\begin{equation}
  \label{eq:near-critical-scaling-intro}
  f^{(k)} \Big( p_\star + \frac{\lambda}{\zeta^k} \Big) \to h^{-1}(\lambda),
\end{equation}
where $h: (0, 1) \to \mathbb{R}$ is a diffeomorphism characterized in Section~\ref{s:near-critical}, where its tail asymptotics are also obtained in terms of critical exponents.

\paragraph{Locality of the critical threshold.}
Finally, Section~\ref{s:locality} deals with the famous conjecture, stating that the threshold $p_c$ should only depend on the local structure of the lattice, while the critical exponents should be determined by the global characteristics of the network (such as the dimension in the Euclidean case).

To state this result, consider the following construction of a hybrid hierarchical lattice:
\begin{enumerate}[\quad a)]
\item Fix two graphs with marked extremes $G_l$ and $G_g$ (standing for local and global respectively) and two integers $l_k$, $g_k$;
\item Start with $G_0$ given by a single edge between $a_0$ and $b_0$;
\item Obtain the graph $G_{k + 1}$ by replacing every edge of $G_k$ with a copy of $G_g$, until $k = g_k$;
\item Repeat the above procedure, now replacing every edge with a copy of $G_l$, for $l_k$ many steps;
\end{enumerate}
\cref{t:locality} roughly states that if $l_k \ll g_k$, then the resulting graph has
\begin{gather}
  \label{eq:locality_intro}
  \text{a threshold $p_k$ converging to the critical value of $G_l$}\\
  \text{a critical window exponent $\zeta_k$ converging to that of $G_g$}.
\end{gather}
See \cref{t:locality} for a more precise statement.

\paragraph{Fixed points of Boolean functions.}
Some of the tools developed in this paper can be applied in more general contexts. As a particularly nice consequence of our results, we obtain a necessary and sufficient condition for the existence of a fixed point $p_{\circ} \in (0,1)$ for the map $p \mapsto \mathbb{E}_p[g]$, where $g:\{0,1\}^n \to \{0,1\}$ is a nontrivial monotone Boolean function that is not the identity. In fact, we prove in Section~\ref{s:fixed_points} that such a point exists if, and only if
\begin{equation}
\frac{\d}{\d p} \mathbb{E}_{p}[g] \Big|_{p=0} = \frac{\d}{\d p} \mathbb{E}_{p}[g] \Big|_{p=1} = 0.
\end{equation}
Furthermore, in this case this fixed point is unique and the derivative of $\mathbb{E}_{p}[g]$ at $p_{\circ}$ is strictly larger than one.

\bigskip

Besides the classical tools from percolation theory, this article has employed a variety of techniques in order to obtain the above described results.
These range from the theory of Dynamical Systems
(see the proofs of \cref{t:crossing,th:beta_and_relations,th:near-critical-scaling,t:locality}),
theory of Boolean Functions (see \cref{t:crossing,t:sharp_ns}),
dynamic Margulis-Russo and recent techniques introduced by Tassion and Vanneuville (see \cref{t:sharp_ns}), theory of multi-type branching processes and results on infinite products of matrices (see the proof of \cref{th:beta_and_relations}).

\subsection{Historical remarks}
\label{ss:history}

Statistical mechanics models have been intensively studied on hierarchical lattices as a way to obtain precise results that would have been impossible to obtain on the Euclidean case.
This has been performed for the Ising model \cite{kaufman,bleher1988limit,Bleher,Khoo2014ACL,ANISIMOVA2021}, for other spin systems \cite{Griffithsb,kaufman1984spin,DeSimoi_2009}, polymers \cite{LACOIN2010467}, and many other systems.

The model of percolation in particular on these lattices has also been studied many times before, such as in \cite{levitan} and \cite{hk2009}, see Remarks~\ref{rmk:off_critical} and~\ref{rmk:one_arm}.

The main contribution of this article lies on the degree of generality and level of precision of the results obtained.
In fact, the results that we state hold for the largest possible class of seed graphs $G_1$, see \cref{rmk:cut_set}.
Moreover the theorems presented cover a vast range of phenomena and their level of precision goes beyond what is currently known on the Euclidean case, even for special cases that are notoriously well understood (such as the triangular lattice and the mean field regime).

Note that our definition of hierarchical lattices differ substantially from those introduced by Dyson for the study of models with long range interactions, see \cite{1969CMaPh..12...91D,brydges1992self,2025arXiv251212124B,bauerschmidt2019introduction,2021arXiv210317013H}.

\subsection{Organization of the paper}

We start by introducing the basic notation of the paper in \cref{sec:notation} and then move on to analyze crossing probabilities in \cref{sec:crossing}, where we establish sharpness and study the off-critical behavior of the system.
In \cref{s:sharp_ns}, we prove that the crossing functions are noise sensitive.
\cref{sec:bensch_limit,sec:bensch_perc} are dedicated to understanding the Benjamini-Schramm limit of these graphs and percolation on the infinite limit.
We then study the critical exponents and their relations in \cref{s:arm_exponents,s:exponent_relations,s:near-critical}.
Finally we prove the locality of the critical threshold and the global behavior of the critical exponents in \cref{s:locality}.
We end the paper with some open problems in \cref{s:open}, with some auxiliary results that were used throughout the text left to \cref{s:auxiliary}.

\subsection{Acknowledgements}

The authors would like to thank Vincent Tassion for the nice observations in Remark~\ref{rmk:tassion}. We thank Senya Shlosman, Yuval Peres, and Gan Chenyu for presenting us the alternative proof of Theorem~\ref{t:crossing} in Remark~\ref{rmk:simpler_proof}. We also thank Gábor Pete for proposing Open Problem~7.

CA was partially supported by the CNPq grant (447397/2024-9).
RB has counted on the support of CNPq grants ``Projeto Universal'' (402952/2023-5), (408529/2025-3), ``Produtividade em Pesquisa'' (308018/2022-2) and ``Jovem Cientista do Nosso Estado'' (204.276/2025) from FAPERJ.
CGM was supported by grants FAPERJ ``Cientista do Nosso Estado'' (200.251/2026) and CNPq ``Projeto Universal'' (408587/2023-7) and ``Produtividade em Pesquisa'' (307594/2022-0).
During this period, AT has been supported by grants ``Projeto Universal'' (408529/2025-3) and ``Produtividade em Pesquisa'' (304437/2018-2 and 304852/2025-2) from CNPq and ``Cientista do Nosso Estado'' (204.377/2024) from FAPERJ.

\section{Notation}
\label{sec:notation}

For a directed graph $H = (V_{H}, E_{H})$ and $x \in V_{H}$, we denote by $\deg^{\mathrm{in}}_{H}(x)$ and $\deg^{\mathrm{out}}_{H}(x)$ the in- and out-degrees of the vertex $x$, respectively.
We write $\deg_{H}(x) = \deg^{\mathrm{in}}_{H}(x) + \deg^{\mathrm{out}}_{H}(x)$ and define the degree of an (oriented) edge $e=(x,y) \in E_{H}$ as $\deg_{H}(e) = \deg_{H}(x)+\deg_{H}(y)$.
When $H$ is connected, we denote by $\d_{H}(\cdot, \cdot)$ the induced graph distance (forgetting orientation) and by $B_{H}(x,r)$ the ball of radius $r$ around $x \in V_{H}$.
For an edge $e=(x,y) \in E_{H}$, we write $B_{H}(e,r) = B_{H}(x,r) \cup B_{H}(y,r)$.
In $G_k$, these quantities will be denoted by $\deg^{\mathrm{in}}_{k}(x)$, $\deg^{\mathrm{out}}_{k}(x)$, $\deg_{k}(\cdot)$, $\d_k(\cdot, \cdot)$, and $B_{k}(\cdot, \cdot)$, respectively.

Given oriented edges $e, e' \in E_{H}$ and a vertex $v \in V_{H}$ in a connected graph $H$, we define the distances between edges, and distance between an edge and vertex as
\begin{equation}
  \begin{split}
  \label{e:dist_edges_def}
  d_{H}(e, e')
    &= \min
      \big\{
        \d_H(v, v'), \, v \text{ endvertex of } e \text{ and } v' \text{ endvertex of } e'
      \big\}, \\
  d_H(e, v)
    &= \min
      \big\{
        \d_H(v, w), \, w \text{ endvertex of } e
      \big\}
  \end{split}.
\end{equation}

Let us precisely introduce the hierarchical graphs we consider in this paper.
We will always convention that $G_0$ is a single edge $(a_0, b_0$, that is $G_{0} = (\{a_0, b_0\}, \{ (a_0, b_0) \})$.
Now, fix a finite, connected, oriented graph $G_1 = (V_1, \Edge_1)$ with two distinguished vertices $a_1, b_1 \in V$.
We define a sequence of oriented graphs $G_k$ in the following way.
Given $G_k = (V_k, \Edge_k)$ with distinguished vertices $a_k, b_k$, the graph $G_{k + 1} = (V_{k + 1}, \Edge_{k + 1})$ is defined as the result of the following procedure: replace each oriented edge $e = (x, y) \in \Edge_k$ by a copy of $G_1$, where the distinguished vertices $a_1$ and $b_1$ are placed on $x$ and $y$, respectively.
Define $a_{k + 1}$ and $b_{k + 1}$ in $G_{k + 1}$ as the vertices that were previously $a_k$ and $b_k$ in $G_k$ before the surgery.
See Figure~\ref{f:d_k} for an illustration of the process.

Given $k \geq 0$, there is a natural correspondence between the edges of $G_k$ and the words of length $k$ with alphabet $\Edge_{1}$.
These words will be called the \emph{address} of an edge.
More precisely, we will define a bijection $\Addr_k: \Edge_k \to \Edge_{1}^k$.
For $k = 0$, the only edge in $G_0$ is associated with the empty word $\varnothing \in \Edge^0$, so $\Addr_0(\{a_0, b_0\}) = \varnothing$.
Suppose that we have defined $A_k$ and recall the procedure to build $G_{k + 1}$ to conclude that every edge $e \in \Edge_{k + 1}$ is associated bi-univocally with a selected edge $e' \in \Edge_k$ and an edge $e'' \in \Edge_1$.
We therefore set $\Addr_{k + 1}(e) = \Addr_k(e') e'' \in \Edge_{1}^{k + 1}$, where we have used the notation of concatenation of a word with a character.

In the same spirit, for $0 \leq j \leq k$ and a word $w \in \Edge_{1}^j$, we define the sub-graph $G_k^w \subseteq G_k$ as the graph induced by the edges $e \in \Edge_k$ such that $w$ is a prefix of $A_k(e)$.
In particular, the following claims hold:
\begin{itemize}
\item $G_k^\varnothing = G_k$,
\item $G_k^w$ for $w \in \Edge_{1}^k$ naturally corresponds to the edges of $G_k$ and
\item for $0 \leq \ell \leq k$, the family $\{G_k^w\}_{w \in \Edge_{1}^{k - \ell}}$ forms an (edge disjoint) paving of the graph $G_k$ by graphs that are isomorphic to $G_\ell$, that is
  \begin{equation}
    \label{eq:decompose}
    G_k = \mcup_{w \in \Edge^{k - \ell}} G_k^w,
  \end{equation}
\end{itemize}

It is also natural to regard the set of words of various lengths $\cup_{k \geq 0} \Edge_{1}^k$ as a tree, where $\varnothing$ is the root and if $w' = we$, then $w'$ is a child of $w$.
In this case, each node of the tree has $|\Edge_{1}|$ children.

\bigskip

Our main goal is to examine the percolative properties of the sequence of graphs $G_k$.
Hence, define, for each $k \geq 1$, the probability space $\Omega_k = [0,1]^{\Edge_k}$ endowed with the Borel $\sigma$-algebra and the product measure $\mathbb{P}_k$ where each marginal is uniformly distributed.
For $p \in [0,1]$, let $\mathbb{P}^{k}_{p}$ denote the distribution of the percolation configuration
\begin{equation}\label{eq:percolation_k}
\omega_p(e) = \textbf{1}_{\omega(e) \leq p}, \text{ for all } \omega \in \Omega_k \text{ and } e \in \Edge_k.
\end{equation}
As usual, an edge $e \in \Edge_k$ is said to be open at level $p$ if $\omega_p(e)=1$ and closed otherwise.

\begin{remark}
  Notice that, although we require the graph $G_{1}$ to be oriented, the percolation model introduced above is non-oriented.
  In the case where there is no graph isomorphism of $G_{1}$ that maps $a_{1}$ into $b_{1}$, this orientation is necessary in order to uniquely define the replacement operation.
  We also do not require the orientation to be consistent, in the sense that it ``flows from $a_{1}$ to $b_{1}$.''
\end{remark}

\section{Crossing Probabilities}
\label{sec:crossing}

In analogy to box crossings for Bernoulli percolation on $\mathbb{Z}^2$, we are going to study the probability that $a_k$ is connected to $b_k$ in $G_k$.
More precisely, let
\begin{gather}
  \label{eq:C_k_f_k}
  C_k = [a_k \longleftrightarrow b_k] \quad \text{(in $G_k$)} \quad \text{and}\\
  f_k:[0, 1] \to [0, 1] \quad \text{given by} \quad f_k(p) = \mathbb{P}^{k}_p (C_k).
\end{gather}
Observe that the probability of any event in $G_k$ is a polynomial in $p$ (with degree at most $|\Edge_k|$), in particular so is $f_k(p)$.

\begin{figure}
  \begin{center}
    \begin{tikzpicture}[scale=4]
      \draw[->] (-.2, 0) -- (1.2, 0) node[right] {$x$};
      \draw[->] (0, -.2) -- (0, 1.2) node[above] {$y$};
      \draw[domain=0:1, smooth, variable=\x, blue] plot ({\x}, {2 *\x*\x - \x*\x*\x*\x});
      \draw[domain=0:1, smooth, variable=\x, red] plot ({\x}, {\x});
      \draw (0.618, 0.618) circle (.01);
      \draw[dashed] (.618, 0) -- (.618, .618);
    \end{tikzpicture}
  \end{center}
  \caption{The function $f_{1}(p) = 2p^{2}-p^{4}$ for the diamond hierarchical lattice in Figure~\ref{f:d_k}.}
  \label{f:function_f}
\end{figure}
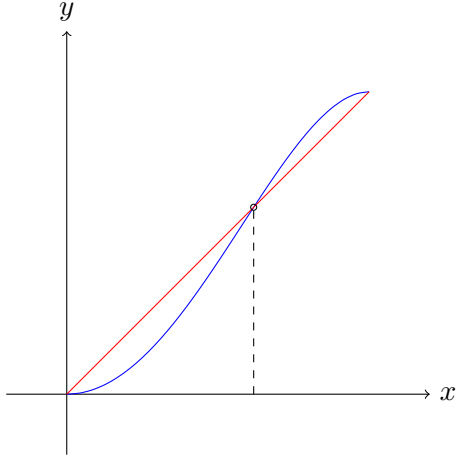

We now make an important observation concerning the recursive nature of the events $C_k$.
For any $k \geq 0$, recalling from \cref{eq:decompose} that $G_{k + 1}$ is given by joining together the graphs $G_k^e$, for $e \in G_1$, we observe that
\begin{display}
  \label{eq:recursive_crossing}
  the event $C_{k+1}$ occurs if and only if there exists a path $\gamma$ from $a_1$ to $b_1$ on $G_1$, such that for each edge $e$ in $\gamma$, the graph $G_k^e$ is crossed.
\end{display}

We now use the above to find a recursive relation between the functions $f_k$.
The special case $k = 1$ is going to be the base of this recursive relation and so we denote $f(p) := f_1(p)$.
From~\cref{eq:recursive_crossing}, one can see that:
\begin{equation}
  \label{e:recursion}
  f_k = f \circ f_{k - 1} = \dots = f^{(k)}.
\end{equation}
Using the relation above, this crossing probability has been studied using techniques from Dynamical Systems in various previous works, see for instance \cite{levitan}.
In particular, one should start by finding the fixed points of $f$.
The result below is the fundamental step in proving~\Cref{t:p_c_constant} and many other results throughout this text.
\begin{theorem}
  \label{t:crossing}
  For any directed graph $G_{1}$ satisfying \cref{e:hypothesis}, we have
  \begin{enumerate}[\quad a)]
  \item ``Uniqueness of phase transition'':\\
    $f$ has exactly three fixed points, namely: $0, 1$ and $p_\star = p_\star(G_1) \in (0, 1)$.
  \item ``The critical point $p_\star$ is a hyperbolic repeller'':\\
    While $f'(0) = f'(1) = 0$,
    \begin{equation}
      \label{eq:lower-bound-zeta}
      \zeta := f'(p_\star) \geq \bigg( 1 - \frac{(1-p_\star)^{\deg(a_1) \vee \deg(b_1)}}{2} \bigg)^{-1} > 1.
    \end{equation}
  \item ``The critical point can be determined in terms of crossing probabilities'':\\
    \begin{equation}
      \label{eq:3}
      p_\star = \sup \{p \in [0, 1]: \lim_k f_k(p) = 0 \}.
    \end{equation}
  \end{enumerate}
\end{theorem}

The proof of the above theorem will be given in a subsection below.

\begin{remark}
  At the critical point $p_{\star}(G_{1})$, it holds that
  \begin{equation}
    \label{e:cross_at_p_star}
    \mathbb{P}^{k}_{p_\star}(C_k) = p_\star, \qquad \text{ for every $k \geq 0$}.
  \end{equation}
\end{remark}

\begin{remark}\label{rmk:cut_set}
Let $\cut(G_{1})$ denote the size of the minimum edge cut-set between $a_{1}$ and $b_{1}$ in $G_{1}$. With this notation, hypothesis~\eqref{e:hypothesis} can be rewritten as $\d_{1}(a_{1}, b_{1})>1$ and $\cut(G_{1}) > 1$. Notice that the crossing probability $f$ always satisfies
\begin{equation}\label{eq:trivial_bounds}
p^{\d_{1}(a_{1},b_{1})} \leq f(p) \leq 1-(1-p)^{\cut(G_{1})}.
\end{equation}
In particular, hypothesis~\eqref{e:hypothesis} states that $f$ presents non-linear behavior near zero and one. Furthermore, the bounds above (and the fact that the inequality on right-hand side above is strict if $\d_{1}(a_{1}, b_{1})>1$) immediately yield
\begin{equation}
p_{\star}(G_{1}) = \begin{cases}
 0, & \qquad \text{if } \d_{1}(a_{1},b_{1})=1, \\
 1, & \qquad \text{if } \cut(G_{1})=1 \text{ and } \d_{1}(a_{1},b_{1}) > 1.
\end{cases}
\end{equation}
As a consequence, hypothesis~\eqref{e:hypothesis} is equivalent to the existence of a non-trivial phase transition.
\end{remark}

\begin{remark}\label{rmk:tassion}
  The two observations below have been made by Vincent Tassion, to whom we are grateful.
  \begin{enumerate}[\quad a)]
  \item Consider the two-dimensional grid $\{1, \dots, n\} \times \{1, \dots, n - 1\}$ with edges connecting nearest neighbor vertices.
    A classical duality argument shows that (under bond percolation at $p = 1/2$) the probability to find a crossing from left to right of this rectangle is $1/2$.
    However, it is not clear: i) if, for every $n \geq 3$, the derivative $\frac{d}{dp} \mathbb{P}_p(\text{cross})\big|_{p = 1/2}$ is strictly larger than one. ii) if, for every $n \geq 3$, $1/2$ is the only fixed point of $p \mapsto \mathbb{P}_p (\text{cross})$.
    Vincent observed that our Theorem~\ref{t:crossing} can be applied in this context (after we collapse the left and right sides of the grid into the vertices $a_1$ and $b_1$), answering both questions above affirmatively.
  \item One possible way to prove item $a)$ of Theorem~\ref{t:crossing} would be to show that the derivative $f'(p)$ is unimodal, meaning that it has a unique local maximum.
    However, as noticed by Vincent again, this is not the case for every graph that satisfies our hypothesis \cref{e:hypothesis}.
    In fact, as a counterexample, we can consider the $n \times (n - 1)$ grid from the above example (with very large $n$), concatenated in series with a diamond.
    The crossing probability for the diamond is $f(p) = 2p^2 - p^4$, see Figure~\ref{f:function_f}, whose derivative has a maximum at a point $\bar{p} > 0.5$.
    The effect of concatenating this graph with a grid in series is to effectively multiply $f$ with the crossing probability $g_n$ of the grid, which converges to the step function $1_{p > 0.5} + 0.5 \cdot 1_{p = 0.5}$, as $n$ grows.
    Thus, the derivative of the resulting product will have at least two modes, one near $1/2$ and another near~$\bar{p}$.
  \end{enumerate}
\end{remark}

For some fixed examples of graphs, such as the diamond shown in Figure~\ref{f:d_k}, the function $f$ is explicit and one can prove \cref{t:crossing} using real analysis.
However, in order to prove the above theorem in its full generality, we introduce an algorithm that determines crossing events, which will later be used in combination with OSSS inequality from~\cite{OSSS}.

\paragraph{Algorithm for crossing events.}

In what follows, we are going to define a family of randomized algorithms $(\mathcal{A}_k)_{k \geq 0}$ that will determine the occurrence of the events $C_k = [a_k \longleftrightarrow b_k]$ in $G_k$.

The definition will be recursive and for the base case, we simply let $\mathcal{A}_0$ be the algorithm that reveals the state of the edge $\{a_0, b_0\}$ in $G_0$.
For $k \geq 0$, we now define $\mathcal{A}_{k + 1}$ in terms of $\mathcal{A}_k$.
Recall $G_{k + 1}$ can be obtained by replacing every edge $e$ of $G_1$ with a copy of $G_k$, which according to the notation introduced in \cref{eq:decompose} is denoted $G_k^e$.
Moreover, by \cref{eq:recursive_crossing}, the crossing of $G_{k + 1}$ can be determined by finding a path of crossed sub-graphs $G_k^e$, linking $a_{k+1}$ to $b_{k+1}$.
We construct $\mathcal{A}_{k + 1}$ as follows:
\begin{enumerate}[\quad i)]
\item Choose a random end-vertex of $G_1$, $X \in \{a_1, b_1\}$ with probability one half for each choice;
\item Run the algorithm $\mathcal{A}_k$ on each graph $G_k^e$ where $e$ is adjacent to $X$ (that is, for all $(G_k^e)_{X \in e}$);
\item If in the above step, all the probed graphs have been determined to be closed (not-crossed), then the algorithm $\mathcal{A}_{k + 1}$ can terminate, since we know that the crossing event $C_{k + 1}$ did not occur.
\item Otherwise, execute the algorithm $\mathcal{A}_k$ on all the remaining sub-graphs $(G_k^e)_{X \not \in e}$.
\end{enumerate}

Let us first note that the algorithms $\mathcal{A}_{k + 1}$ defined above indeed determine the occurrence of $C_{k + 1}$.
In fact, if the graphs probed in step $ii)$ are all closed, then, by \cref{eq:recursive_crossing}, the event $C_{k+1}$ cannot hold, since any crossing will necessarily include a crossing of some $G_k^e$ for $e$ containing $X$.
Alternatively, if that was not the case, the algorithm $\mathcal{A}_{k + 1}$ is indeed capable of inferring the occurrence of $C_{k + 1}$, since this event is measurable with respect to $\sigma({\bf 1}\{G_k^e \text{ is crossed}\}: e \in \Edge_{1})$, again using \cref{eq:recursive_crossing}.

In order to employ the OSSS inequality, we need to analyze the revealment of the algorithms above.
For any edge $e \in \Edge_{k}$, let
\begin{equation}
  \label{e:max_revealment}
  \delta_{e}(p) := \mathbb{P}_p^{k} \big( \text{$e$ is revealed during the execution of $\mathcal{A}_k$} \big),
\end{equation}
where the above probability also includes the extra randomness used in the definition of $\mathcal{A}_k$.

\begin{lemma}
  \label{l:revealment}
  If $p_\circ$ is a fixed point of $f$, then, for any $k \geq 0$,
  \begin{equation}
    \label{e:revealment}
    \max_{\text{$e \in \Edge_k$}} \delta_e(p_{\circ})
    \leq \bigg( 1 - \frac{(1-p_\circ)^{\deg(a_1)\vee\deg(b_1)}}{2} \bigg)^k.
  \end{equation}
\end{lemma}

\begin{proof}
  The proof will follow by induction in $k$, with the base case $k = 0$ being obvious.
  Suppose now that \cref{e:revealment} is true for $k \geq 0$ and fix some edge $e_\star \in \Edge_{k + 1}$.

  Recalling that $\d_1(a_1, b_1) \geq 2$, the family of graphs $(G_k^e)_{a_{1} \in e}$ is disjoint from the family $(G_k^e)_{b_{1} \in e}$.
  This means that our fixed edge $e_\star$ cannot belong to both families.
  Suppose first that it belongs to the first family $(G_k^e)_{a_{1} \in e}$ so that
  \begin{equation}
    \begin{split}
      \mathbb{P}^{k + 1}_{p_{\circ}} \big(
      & \text{$e_\star$ is revealed by $\mathcal{A}_{k + 1}$}
        \big)
        = \mathbb{P}^{k + 1}_{p_{\circ}}(X = a_{1})
        \mathbb{P}^k_{p_{\circ}} \big( \text{$e_\star$ is revealed by $\mathcal{A}_k$} \big)\\
      & \quad + \mathbb{P}^{k + 1}_{p_{\circ}}(X = b_{1})
        \mathbb{P}^k_{p_{\circ}} \big( \text{some graph $G_k^e$ with $b \in e$ is crossed} \big)\\
      & \quad \quad \; \mathbb{P}^k_{p_{\circ}} \big( \text{$e_\star$ is revealed by $\mathcal{A}_k$} \big)\\
      & \leq \bigg( 1 - \frac{(1-p_\circ)^{\deg(a_1)\vee\deg(b_1)}}{2} \bigg)^k
        \frac{1}{2} \Big( 1 + \big( 1 - (1-p_\circ)^{\deg(b_1)} \big) \Big)\\
      & \leq \bigg( 1 - \frac{(1-p_\circ)^{\deg(a_1)\vee\deg(b_1)}}{2} \bigg)^{k + 1}.
    \end{split}
  \end{equation}
  The same calculation can be performed if $e_\star$ belongs to the other family $(G_k^e)_{b \in e}$ or if it does not belong to either.
  In any case, since the edge $e_\star$ was arbitrarily chosen, we have concluded the proof of \Cref{l:revealment}.
\end{proof}

\paragraph{Proof of \cref{t:crossing}.}

Given a configuration $\omega_{k}$ and an edge $e \in G_{k}$, define the configuration $\omega_{k}^{e}$ by changing the configuration $\omega_{k}$ only at $e$. We say that $e$ is pivotal for the pair $(C_{k}, \omega_{k})$ if
\begin{equation}\label{eq:def_pivotal}
\textbf{1}_{C_{k}} (\omega_{k}) \neq \textbf{1}_{C_{k}} (\omega_{k}^{e}).
\end{equation}
We now turn to the proof of the theorem.

\begin{proof}[Proof of \cref{t:crossing}]
  Concerning point $a)$, we start with the trivial observation that $0$ and $1$ are fixed points for the crossing function $f$, since in these cases the occurrence of a crossing is deterministically established.
  Notice that $f$ is infinitely differentiable as it is a polynomial in $p$.
  Moreover, by Russo's formula (see Theorem~2.25 in \cite{Gri99}),
  \begin{equation}
    \label{eq:derivative_f}
    \frac{df}{dp} = \mathbb{E}_p \big[ \# \{ \text{pivotal edges to $C_1$} \} \big],
  \end{equation}
  which is clearly zero for $p \in \{0, 1\}$, since $\d_{1}(a_1, b_1) \geq 2$ by \cref{e:hypothesis}.
  This fact implies the existence of at least one other fixed point of $f$ in $(0, 1)$.
  In order to prove its uniqueness, it will be enough to show that for any $p_\circ \in (0, 1)$ such that $f(p_\circ) = p_\circ$, we have $f'(p_\circ) > 1$, which we will do by proving the bound \cref{eq:lower-bound-zeta} for $f'(p_\circ)$.
  This will show not only the uniqueness of the fixed point in $(0, 1)$, but also $b)$.

  Suppose thus that $p_\circ \in (0, 1)$ is such that $f(p_\circ) = p_\circ$ and observe by the chain rule that $f'_k(p_\circ) = (f^{(k)})'(p_\circ) = (f'(p_\circ))^k$.
  We recall the OSSS inequality (see~\cite{OSSS}) which states that for any algorithm $\mathcal{A}_k$ revealing $f_k$, we have
  \begin{equation}
    \label{eq:osss}
    \Var_{p_{\circ}}(f_k) \leq \sum_{e \in \Edge_k} \delta_e(p_{\circ}) \mathbb{P}_{p_\circ} (e \text{ is pivotal to $C_k$}).
  \end{equation}
  From the above and again by Russo's formula, we obtain
  \begin{equation}
    \label{eq:largederivative}
    \begin{split}
      f'_k(p_\circ) & = \sum_{e \in G_k} \mathbb{P}_{p_\circ} (e \text{ is pivotal to $C_k$})
      \geq \Var_{p_{\circ}}(f_k) \frac{1}{\max_e \delta_e(p_{\circ})} \\
      & \geq p_\circ (1 - p_\circ) \bigg( 1 - \frac{(1-p_\circ)^{\deg(a_1)\vee\deg(b_1)}}{2} \bigg)^{-k},
    \end{split}
  \end{equation}
  so that $f'(p_\circ) = \liminf f_k'(p_\circ)^{1/k}$ satisfies \cref{eq:lower-bound-zeta}, finishing the proof of $a)$ and $b)$.

  We now turn to the proof of $c)$ and we start by noting that $f|_{[0, p_\star]}$ smaller than or equal to the identity function on the same interval.
  Russo's formula implies that, for every $p < p_\star$, $f^{(k)}(p)$ is positive and strictly increasing.
  This, together with the fact that $f$ is continuous and its only fixed point below $p_\star$ is zero, implies that $\lim_k f_k(p) = 0$.
  A similar argument shows that every $p > p_\star$ satisfies $\lim_k f_k(p) = 1$, showing $c)$.
\end{proof}

\begin{remark}\label{rmk:simpler_proof}
  After one of the authors presented this result at a conference, he was approached by Senya Shlosman, intrigued by the fact that the statement of \cref{t:crossing} is not related to hierarchical lattices (but rather about Boolean functions), while the proof we present above involves different scales $k$.
  Later he contacted Yuval Peres, who, together with Gan Chenyu, came up with a simpler proof of the statement that we now present.

  Consider the algorithm $\mathcal{A}_1$ defined above and recall the strengthened version of the OSSS inequality:
  \begin{equation}
    \Var_{p_\circ}(f)
    \leq p_\circ (1 - p_\circ) \sum_{e \in G_1}
    \delta_e(p_\circ) \mathbb{P}_{p_\circ} (\text{$e$ is pivotal to $C_1$}),
  \end{equation}
  see \cite{OSSS} (noting that their Boolean functions take values in $\{-1, 1\}$ instead of $\{0, 1\}$) and the observation in page 2, after the definition of influence.
  On the other hand, by the assumption that $p_\circ$ is a fixed point of $f$, $\Var_{p_\circ}(f) = p_\circ (1 - p_\circ)$, so that we can obtain the simplified bound $f'(p_\circ) \geq 1/(\max_{\text{$e \in \Edge_1$}} \delta_e(p_{\circ}))$, which is larger then one by~\eqref{e:revealment}.

  We are grateful to Senya, Yuval, and Gan for this remarkable simplification.
\end{remark}

\subsection{Off-critical behavior and correlation length}

  We now consider the off-critical behavior of the crossing probabilities $f_k$.
  In this section we analyze two different quantities of interest.
  First, in the sub-critical phase, we examine the correlation length, while for the super-critical phase we investigate the surface tension.
  We explicitly obtain the critical exponents for both quantities.

  Roughly speaking, these two quantities measure how far ones needs to observe the configuration in order to detect a behavior that is characteristic of being away from criticality.
  For $p < p_c$, this would mean to look at a large enough box, so that only ``mesoscopic clusters'' are present.
  While for $p > p_c$, looking at a sufficiently large box guarantees that only one giant component exists and it touches all ``mesoscopic sub-boxes''.

\medskip

Recall from Remark~\ref{rmk:cut_set} that $\cut(G_{1})$ denotes the size of the minimum edge cut-set between $a_{1}$ and $b_{1}$ in $G_{1}$.
\nc{c:corr_exp_1}
\nc{c:corr_exp_2}
\nc{c:surf_exp_1}
\nc{c:surf_exp_2}
\begin{theorem}
  \label{th:off-critical-decay}
  For any directed graph $G$ satisfying \cref{e:hypothesis}, we have
  \begin{enumerate}[\quad a)]
  \item \label{th:off-critical-existence}
    (Existence of correlation length and surface tension) The function $\xi: (0, p_\star) \to (0, \infty)$ given by
    \begin{equation}
      \label{eq:off-critical-decay-sub}
      \xi(p) = - \lim_k \frac{\log f_k(p) }{\d_{1}(a_1, b_1)^k}
    \end{equation}
    is well-defined and non-decreasing.
    The non-increasing function $\eta: (p_\star, 1) \to (0, \infty)$,
    \begin{equation}
      \label{eq:off-critical-decay-super}
      \eta(p) = - \lim_k \frac{\log \big( 1- f_k(p) \big)}{\cut(G_{1})^k}
    \end{equation}
    is well-defined.
  \item \label{th:off-critical-exponent}
    (Critical exponents for the correlation length and surface tension) There exist constants $\uc{c:corr_exp_1}, \uc{c:corr_exp_2} > 0$ such that
    \begin{equation}
      \label{eq:critical_exponent_correlation}
      \uc{c:corr_exp_1} (p_\star - p)^{-\nu}
      \leq \xi(p)
      \leq \uc{c:corr_exp_2} (p_\star - p)^{-\nu},
      \text{ for $p \in (p_\star/2, p_\star)$}
    \end{equation}
    where $\nu = \log_\zeta \d_{1}(a_1, b_1)$.
    Moreover, there exist constants $\uc{c:surf_exp_1}, \uc{c:surf_exp_2} > 0$ such that
    \begin{equation}
      \label{eq:critical_exponent_surface}
      \uc{c:surf_exp_1} (p - p_\star)^{-\mu}
      \leq \eta(p)
      \leq \uc{c:surf_exp_2} (p - p_\star)^{-\mu},
      \text{ for $p \in (p_\star, 1 - (1-p_\star)/2)$}
    \end{equation}
    with $\mu = \log_\zeta \cut(G_{1})$.
  \end{enumerate}
\end{theorem}

  \begin{remark}\label{rmk:off_critical}
  Regarding the result above, we note the following.
    \begin{enumerate}
    \item In Levitan~\cite{levitan}, the computation of the critical exponent $\nu$ was carried out for the diamond hierarchical lattice (see Figure~\ref{f:d_k}), obtaining $\nu = \frac{1}{\log_{2}(6-2\sqrt{5})} \approx 1.6352$.
    \item It is widely believed that macropcopic observables of percolation, such as the correlation length, the surface tension, and the percolation probability, should behave as powers of $|p-p_{c}|$ in the neighborhood of the critical points, and the exponents controlling these decays obey the so-called ``scaling relations.'' While these exponents are conjectured to exist for Bernoulli percolation in $\Z^{d}$ for every dimension $d$, this still remains open. Kesten~\cite{kesten1987scaling, kesten1987scaling2} proved that, provided such exponents do exist for $d=2$, they satisfy the ``scaling relations.'' The existence of these exponents was established for site percolation in the triangular lattice, see for example the lecture notes~\cite{werner2007lectures} and references therein. More recently, in a series of three papers, Hutchcroft~\cite{hutchcroft2025critical1, hutchcroft2025critical2, hutchcroft2025critical3} established the validity of several scaling relations for long-range percolation in $\Z^{d}$. For a more thorough account of the critical exponents and their scaling relations we refer the reader to~\cite{Gri99}.
    \item While, in the case of the triangular lattice mentioned above, the behavior of quantities such as $\xi(p)$ or $\eta(p)$ is obtained up to smaller order multiplicative errors, meaning that~\eqref{eq:critical_exponent_correlation} is stated as $\xi(p) \sim (p-p_{c})^{-\nu+o(1)}$ (see for example~\cite[Chapter 6]{werner2007lectures}), with a similar expression for~\eqref{eq:critical_exponent_surface}, our result is sharper in the sense that it holds without the lower order corrections.
    \end{enumerate}
  \end{remark}

\begin{proof}
  \nc{c:pert_d}
  In order to prove \cref{eq:off-critical-decay-sub}, we start by observing that
  \begin{equation}
    \label{eq:derivatives_vanish}
    0 = f(0) = \frac{d f}{d p}(0) = \dots = \frac{d^{r-1} f}{d p^{r - 1}}(0),
  \end{equation}
  where $r = \d_{1}(a_1, b_1)$, while $d^r f/d p^r > 0$.
  This can be obtained by either a generalization of Russo's formula (see Theorem~1.5 of \cite{ARPE_MOSSEL_2010}) or by observing that the terms with degree $r$ in the polynomial $f(p)$ are given exactly by the geodesic paths from $a_1$ to $b_1$.

  With this, we use Taylor's expansion to write
  \begin{equation}
    \label{eq:taylor_f_zero}
    f(p) = F(p) p^r
  \end{equation}
  where $F: [0, p_\star] \to (0, \infty)$ is continuous.
  As for the iterations
  \begin{equation}
    \frac{\log f_{k + 1}(p)}{r^{k + 1}}
    = r \frac{\log f_{k}(p)}{r^{k + 1}}
    + \frac{\log F(f_{k}(p))}{r^{k + 1}},
  \end{equation}
  so that we can write a telescopic series
  \begin{equation}
    \frac{\log f_{k}(p)}{r^k}
    = \log(p) + \sum_{j = 0}^{k - 1} \frac{1}{r^{j+1}}
    \log F \big( f_{j}(p) \big),
  \end{equation}
  which clearly converges, since $F$ is bounded away from zero and infinity.
  This proves that the limit defining $\xi(p)$ in \cref{eq:off-critical-decay-sub} exists, the monotonicity of $\xi$ following directly from that of $f$.
  The proof that the limit in \cref{eq:off-critical-decay-super} exists follows exactly the same steps, but analyzing $1-f$ in $[p_\star, 1]$ instead.

  We now turn to the proof of \cref{eq:critical_exponent_correlation}.
  Since $p_\star$ is a hyperbolic fixed point of $f$, we can use Hartman-Grobman's Theorem (see Appendix~\ref{app:auxiliary}) to obtain a neighborhood $(a, b)$ of $p_\star$, a neighborhood $(c, d)$ of zero, and an increasing diffeomorphism $h:(a, b) \to (c, d)$ such that
  \begin{equation}
    \label{eq:hg-correlation}
    h \circ f \circ h^{-1}(x) = \zeta x, \text{ for $x \in (c, d)$}.
  \end{equation}
  In other words, $f$ conjugates locally with a linear map.
  Without loss of generality, we can assume that $a > 0$ and since $p < p_\star$ we get by \cref{eq:3} that
  \begin{equation}
    k(p) := \inf \{k \geq 0: f_{k}(p) < a\}
  \end{equation}
  is finite almost surely.
  The behavior of $k(p)$ will control the early steps of the dynamics and one can use \cref{eq:hg-correlation} to understand the asymptotic behavior of $k(p)$ as $p$ goes to $p_\star$.
  More precisely,
  \begin{equation}
    \label{eq:k_p_with_hg}
    k(p) := \inf \{k \geq 0: \zeta^{k} h(p) < c \}
    = \Big\lceil \log_\zeta(- c) - \log_\zeta \big( -h(p) \big) \Big\rceil
  \end{equation}
  and we can now turn to contraction stage of the dynamics.

  Since $\xi$ is monotone, we see that, for $k \geq k(p)$,
  \begin{equation}
    \frac{\log f_{k - k(p) + 1}(a)}{\d_{1}(a_1, b_1)^k}
    \leq \frac{\log f_{k}(p)}{\d_{1}(a_1, b_1)^k}
    \leq \frac{\log f_{k - k(p)}(a)}{\d_{1}(a_1, b_1)^k}.
  \end{equation}
  Taking the limit on both sides yields
  \begin{equation}
    \frac{\xi \big( f(a) \big)}{\d_{1}(a_1, b_1)^{k(p)}}
    \leq \xi(p)
    \leq \frac{\xi \big( a \big)}{\d_{1}(a_1, b_1)^{k(p)}}.
  \end{equation}
  Plugging \cref{eq:k_p_with_hg} in the equation above and using the fact that the conjugation $h$ is differentiable at $p_{\star}$ and can be chosen so that $h'(p_{\star}) = 1$ (see Proposition~\ref{p:diffeo_conjugation} in Appendix~\ref{app:auxiliary}) implies \cref{eq:critical_exponent_correlation}.
  The proof of \eqref{eq:critical_exponent_surface} is analogous and we omit it here.
\end{proof}

\section{Sharp noise sensitivity}
\label{s:sharp_ns}

In this section, we work with enlarged probability spaces $\tilde{\Omega}_{k} = [0,1]^{\Edge_{k}} \times [0,1]^{\Edge_{k}} \times [0,1]^{\Edge_{k}}$ and consider the probability measure $\mathbb{P}_{k}$ in $\tilde{\Omega}_{k}$ where each marginal is uniformly distributed and independent.
Associate with $e \in \Edge_{k}$ a triple $(U_{e}, \tilde{U}_{e}, U_{e}') \in [0,1]^{3}$.
Given $p \in [0,1]$ and $\varepsilon \in [0,1]$, define the configurations
\begin{equation}\label{eq:resampling}
  \omega_{p}(e) = \textbf{1}_{U_{e} \leq p}
  \quad \text{and} \quad
  \omega_{p}^{\varepsilon}(e) =
  \begin{cases}
    \textbf{1}_{\tilde{U}_{e} \leq p}, & \quad \text{if } U_{e}' \leq \varepsilon, \\
    \omega_{p}(e), & \quad \text{otherwise}.
  \end{cases}
\end{equation}
In words, the fields $\omega_{p}^{\varepsilon}$ is obtained as an $\varepsilon$-resampling of the field $\omega_{p}$.
Although these configurations are not independent from each other, they have the same distribution. Furthermore, under $\mathbb{P}_{k}$, the configuration $\omega_{p}$ has the same distribution as the original configuration $\omega_{p}$ introduced in~\eqref{eq:percolation_k}, measured with respect to $\mathbb{P}_{p}^{k}$.

\begin{theorem}\label{t:sharp_ns}
  There exists a non-decreasing homeomorphism $g: [0,+\infty] \to [0,1]$ such that, if $\varepsilon_{k} \in [0,1]$ satisfies $\lim_{k} \zeta^{k} \varepsilon_{k} = c \in [0, +\infty]$, then
  \begin{equation}
    \lim_{k \to \infty} \mathbb{E}_{k} \big[
    \textnormal{\textbf{1}}_{C_{k}}(\omega_{p_{\star}})
    \textnormal{\textbf{1}}_{C_{k}}(\omega_{p_{\star}}^{\varepsilon_{k}})\big]
    = p_{\star}^{2} g(c)+ p_{\star}(1-g(c)).
  \end{equation}
\end{theorem}

\begin{remark}
  The theorem above implies that the sequence of Boolean functions $\textbf{1}_{C_{k}}$ is noise sensitive when $\zeta^{k} \varepsilon_{k} \to \infty$.
\end{remark}

For a Booelan function $\mathfrak{f} : \{0,1\}^{n} \to \{0,1\}$, define the functional
\begin{equation}\label{eq:operator_Q}
  Q_{\varepsilon}^{p}(\mathfrak{f})
  = \mathbb{E}_{p} \big[\mathfrak{f}(\omega)
  \mathfrak{f}(\omega^{\varepsilon}) \big],
\end{equation}
where $\mathbb{E}_{p}$ denotes the expectation with respect to the probability distribution such that $\omega \in \{0,1\}^{n}$ has i.i.d.\ entries with distribution $\Ber(p)$ and $\omega^{\varepsilon}$ is obtained as an $\varepsilon$-resampling of $\omega$ (defined analogously to~\eqref{eq:resampling}).

\begin{proposition}[Dynamic Margulis-Russo, Tassion and Vanneuville~\cite{tv2023}]
  \label{prop:dynamic_margulis_russo}
  Given any Boolean function $\mathfrak{f}: \{0,1\}^{n} \to \{0,1\}$, the functional $Q_{\varepsilon}(\mathfrak{f})$ is differentiable in $\varepsilon$ and
  \begin{equation}
    0 \leq -\frac{d}{d \varepsilon} Q_{\varepsilon}^{p}(\mathfrak{f})
    \leq p(1-p)\sum_{i \in [n]} Q_{\varepsilon}^{p}
    \big( i \textnormal{ is pivotal for } \mathfrak{f} \big).
  \end{equation}
  Furthermore, if $\mathfrak{f}$ is increasing, the second inequality is an equality.
\end{proposition}

\begin{remark}
  Although the proposition above is proved in~\cite{tv2023} only for the case $p = \frac{1}{2}$, the same proof strategy can be applied to different values of $p \in [0,1]$.
\end{remark}

We are now ready to prove Theorem~\ref{t:sharp_ns}.
\begin{proof}[Proof of Theorem~\ref{t:sharp_ns}]
Notice that, for every $k \geq 1$,
\begin{equation}\label{eq:recursive_ns}
  \mathbb{E}_{k} \big[ \textnormal{\textbf{1}}_{C_{k}}(\omega_{p_{\star}})
  \textnormal{\textbf{1}}_{C_{k}}(\omega_{p_{\star}}^{\varepsilon})\big]
  = \mathbb{E}_{k-1} \big[
  \textnormal{\textbf{1}}_{C_{k-1}}(\omega_{p_{\star}})
  \textnormal{\textbf{1}}_{C_{k-1}}(\omega_{p_{\star}}^{s(\varepsilon)})\big]
\end{equation}
where the sensitivity function $s(\varepsilon)$ is given by
\begin{equation}
  s(\varepsilon) = 1-\frac{\mathbb{E}_{1} \big[
    \textnormal{\textbf{1}}_{C_{1}}(\omega_{p_{\star}})
    \textnormal{\textbf{1}}_{C_{1}}(\omega_{p_{\star}}^{\varepsilon})\big]
    - p_{\star}^{2}}{p_{\star}-p_{\star}^{2}}.
\end{equation}
This follows from the recursive definition of the graphs $G_{k}$ and~\eqref{eq:recursive_crossing}.

Applying~\eqref{eq:recursive_ns} repeatedly yields
\begin{equation}
  \begin{split}
    \mathbb{E}_{k} \big[ \textnormal{\textbf{1}}_{C_{k}}(\omega_{p_{\star}})
    \textnormal{\textbf{1}}_{C_{k}}(\omega_{p_{\star}}^{\varepsilon_{k}})\big]
    & = \mathbb{E}_{0} \big[ \textnormal{\textbf{1}}_{C_{0}}(\omega_{p_{\star}})
      \textnormal{\textbf{1}}_{C_{0}} \big(
      \omega_{p_{\star}}^{s^{k}(\varepsilon_{k})} \big) \big] \\
    & = p_{\star}^{2}s^{k}(\varepsilon_{k})+p_{\star}(1-s^{k}(\varepsilon_{k})).
  \end{split}
\end{equation}

The result is thus complete if we verify that
\begin{equation}
  \lim_{k \to \infty} s^{k}(\varepsilon_{k}) = g(c).
\end{equation}

In order to do so, first observe that $s(0)=0$ and $s(1) = 1$. We will now check that these are the only fixed points of the function $s$.

Borrowing the notation from~\eqref{eq:operator_Q}, we can write
\begin{equation}
  s(\varepsilon) = 1 - \frac{Q^{p_{\star}}_{\varepsilon}
    (\textbf{1}_{C_{1}})-p_{\star}^{2}}{p_{\star}-p_{\star}^{2}}.
\end{equation}

Applying Proposition~\ref{prop:dynamic_margulis_russo} to the function $s$ yields
\begin{equation}
  s'(0) = \mathbb{E}_{p_{\star}} \big[ \#\{ \text{pivotal edges to $C_1$} \} \big]
  = \zeta.
\end{equation}
Furthermore, repeated application of Proposition~\ref{prop:dynamic_margulis_russo} implies that $s''(\varepsilon) \leq 0$, for all $\varepsilon \in [0,1]$.
This in particular implies that $s:[0,1] \to [0,1]$ is a non-decreasing concave polynomial whose only fixed points in the interval $[0,1]$ are the points $0$ and $1$.

We now apply Hartman-Grobman Theorem, the map $s$ can be conjugated to a linear map around the origin.
As per Appendix~\ref{app:auxiliary}, since $s$ has no fixed points in the open interval $(0,1)$, the conjugancy can be extended to the whole interval $[0,1)$, that is, there exists an homeomorphism $r:[0,1) \to [0,+\infty)$ such that
\begin{equation}
  s(\varepsilon) = r^{-1} \big( \zeta r(\varepsilon) \big),
\end{equation}
for all $\varepsilon \in [0,1)$.

Since $s$ is a polynomial, the map $r$ is differentiable at $0$ and can be chosen so that $r'(0)=1$ (see the comment after Proposition~\ref{p:diffeo_conjugation}).
First-order Taylor expansion for the function $r$ yields
\begin{equation}
  r(\varepsilon) = \varepsilon + O(\varepsilon^{2}).
\end{equation}

In particular, we obtain
\begin{equation}
  \lim_{k \to \infty} s^{k}(\varepsilon_{k})
  = \lim_{k \to \infty} r^{-1} \big( \zeta^{k} r(\varepsilon_{k}) \big)
  = \lim_{k \to \infty} r^{-1} \big( \zeta^{k} \varepsilon_{k}
  + \zeta^{k}O(\varepsilon_{k}^{2}) \big) = r^{-1}(c).
\end{equation}
This concludes the proof with $g = r^{-1}$.
\end{proof}

\section{The Benjamini-Schramm limit}
\label{sec:bensch_limit}

This section will be more technical in nature, and the proofs can be omitted without loss of clarity to the reader. The main result to be used in subsequent sections is Proposition~\ref{p:benjamini_schramm_marks}, about properties of the local limit of hierarchical graphs.

We say two locally finite, connected graphs endowed with distinguished edges $(G, e)$ and $(G', e')$ are equivalent if there exists a graph isomorphism between them mapping $e$ to $e'$.
Let $\mathcal{G}$ be the space of all connected, locally-finite, edge-rooted graphs, modulo this equivalence.
From now on, we refer to an edge-rooted graph $\bar{G} = (G, e)$ and its equivalence class in $\mathcal{G}$ indistinguishably.

Given $\bar{G} = (G, e) \in \mathcal{G}$ and $r \geq 0$, we write $\bar{G}\big|_r \in \mathcal{G}$ for the rotted graph induced by the ball of radius $r$ around the root edge $e$. That is, the vertices of $\bar{G}\big|_r$ are all vertices in $G$ within distance $r$ to any endpoint of $e$.

Given two edge-rooted graphs $\bar{G} = (G, e), \bar{G}' = (G', e') \in \mathcal{G}$, their distance is defined as
\begin{equation}
  \label{e:distance_G}
  \d(\bar{G}, \bar{G}') := \inf \Big\{
  \frac{1}{r + 1}: \, \bar{G}\big|_r \text{ is isomorphic to } \bar{G}'\big|_r
  \Big\}.
\end{equation}
It can be shown that $(\mathcal{G}, d)$ is a Polish metric space, see \cite{Curien2017}.

What follows is the notion of convergence introduced by Benjamini and Schramm, see \cite{benjamini2011recurrence}, but in this case we are considering a distinguished edge instead of a distinguished vertex.
We will need to define the empirical probability measure associated with a finite, connected graph $G = (V, \Edge)$:
\begin{equation}
  \label{e:empirical}
  \mu := \frac{1}{|\Edge|} \sum_{e \in \Edge} \delta_{(V, \Edge, e)},
\end{equation}
which is a probability measure in $\mathcal{G}$, obtained by rooting $G$ on a random edge uniformly selected from the set $\Edge$.

\begin{definition}
  \label{d:besh}
  A sequence of finite connected graphs $G_n = (V_n, \Edge_n)$ converges in the Benjamini-Schramm topology to a random element $\bar{G}_\infty = (V_\infty, \Edge_\infty, E_\infty) \in \mathcal{G}$ with distribution $\mu$ if $\mu_n$ converges weakly to $\mu$.
\end{definition}

Recall the definition of the hierarchical graphs $(G_k)$ in \Cref{sec:notation}.
We define $\bar{G}_k := (G_k, E_k)$ as the graph $G_k$ rooted on a uniformly selected edge $E_k \in \Edge_k$, forgetting the edge orientations, and denote by $\mu_k$ the distribution of $\bar{G}_{k}$.
In the following, we obtain the Benjamini-Schramm limit $(G_\infty, E_\infty)$ of the sequence $\bar{G}_{k}$, and analyze its local neighborhoods in terms of graphs of the form $G_k$.
To that end, recall the \emph{address} of an edge, with which we will mark each edge in the graph, that is, the bijections $A_k: \Edge_k \to \Edge^k$ defined in \Cref{sec:notation}. Abstractly, we will construct a space where we can mark arbitrary edges with addresses of arbitrary length, including infinity.

As we investigate hierarchical graphs, sometimes it is useful to analyze them by zooming out, from the smallest graphs to the largest ones, and sometimes the opposite (top-down view) is more fruitful.
The first approach is necessary in the local convergence context -- we need the smallest graphs in the hierarchical construction to be stable.
Apologizing in advance for the change of perspective, we consider the set of addresses indexed by negative integers, identifying $\Edge^k$ with $\Edge_1^{(-k, \dots, -1)}$, so that edges associate to the smaller graphs are closer to $0$. We let
\begin{equation}
  \label{e:address_def}
  \mathcal{A}
    :=
      \Edge_1^{\mathbb{Z}_{-}} \cup \bigcup_{k \geq 1} \Edge_1^{(-k, \dots, -1)}.
\end{equation}

There exists a natural notion of restriction in $\mathcal{A}$, namely, if $r > 0$ and $w \in \mathcal{A}$, we let $w \big|_r$ be the sequence composed by the \emph{last} $r$ elements in $w$, if the length of $w$ is at least $r$, or simply $w$ otherwise. That is, given
\begin{equation}
  \label{e:address_restriction_def}
  w = e_{-k} e_{-(k-1)}\dots e_{-1},
\end{equation}
we let
\begin{equation}
  \label{e:address_restriction_def2}
  w\big|_r = 
  \begin{cases}
    e_{-r} e_{-(r-1)} \dots e_{-1},  & \quad \text{if } k \geq r,
    \\
    w, & \quad \text{otherwise.}
  \end{cases}
\end{equation}
We can therefore define for $w, w' \in \mathcal{A}$, similarly to \Cref{e:distance_G},
\begin{equation}
  \label{e:distance_address}
  \d(w, w') := \inf \Big\{
    \frac{1}{r + 1}: \, w\big|_r = w'\big|_r
    \Big\}.
\end{equation}
Since $\Edge_1$ is a finite set, it is elementary to see that $\mathcal{A}$ with this metric is totally bounded and complete, hence compact.
For a given graph with edge set $\Edge$, $r > 0$, and a function defining the edge marks $\mathsf{m}: \Edge \to \mathcal{A}$, we can define the $r$-restriction of $\mathsf{m}$ simply by edge-wise mark restriction, that is,
\begin{equation}
  \label{e:mark_restriction}
  \mathsf{m}\big|_r(e) = \mathsf{m}(e)\big|_r
\end{equation}
for every $e \in \Edge$.
We then consider the space $\mathcal{G}_\mathcal{A}$ of connected edge-rooted graphs with marks in $\mathcal{A}$, that is, $(G, e, \mathsf{m}) \in \mathcal{G}_\mathcal{A}$ if $(G, e) \in \mathcal{G}$ and $\mathsf{m}: \Edge \to \mathcal{A}$.
This space also has a natural notion of local topology
\begin{equation}
  \label{e:distance_G_A}
  \begin{split}
    \d(&(G, e, \mathsf{m}) , (G', e', \mathsf{m}'))
    \\
      &:=
        \inf
          \Bigg\{
            \frac{1}{r+1} + \frac{1}{k+1}:
            \begin{array}{c}
            \text{ there exists a root-preserving isomorphism} \\
            \text{ between }
            (G, e)\big|_r \text{ and } (G', e')\big|_r
            \text{ preserving} \\
            \text{ the marks }
            \mathsf{m}\big|_k \text{ and } \mathsf{m}'\big|_k
            \end{array}
          \Bigg\}.
  \end{split}
\end{equation}

The following proposition shows a tightness criterion for this convergence.
\begin{proposition}
  \label{p:tightness}
  Consider a family of random edge-rooted graphs $(G^{\lambda}, E^{\lambda})_{\lambda \in \Lambda}$ in $\mathcal{G}$, and let $\deg(E^\lambda)$ denote the sum of degrees of endpoints of $E^{\lambda}$. This family is tight if and only if
  \begin{equation}
    \label{e:tightness}
    \deg(E^{\lambda}), \text{ for $\lambda \in \Lambda$},
  \end{equation}
  is tight. The same result also holds for any family of random edge-rooted graphs with marks in $\mathcal{A}$, $(G^{\lambda}, E^{\lambda}, \mathsf{m}^\lambda)_{\lambda \in \Lambda}$ in $\mathcal{G}_\mathcal{A}$.
\end{proposition}

\begin{proof}
  The first statement follows from the tighness criterion for stationary random graphs in Proposition 21 of~\cite{Curien2017}.
  These are graphs with one distinguished directed edge, and we can change from edge-rooted to directed-edge rooted by forgetting the orientation of the rooted edge. 

  Lemma $A.1$ from~\cite{baldasso2022large} shows that taking a compact set in $\mathcal{G}$ and considering all possible marks on graphs of this set yields a compact set in the local topology of $\mathcal{G}_{\mathcal{A}}$, since the collection of marks $\mathcal{A}$ is compact. This fact, together with the first statement, proves the second the result for marked edge-rooted graphs.
\end{proof}

For each edge $e \in \Edge_k$, we mark it with its (reversed) address $A_k(e)$, thus obtaining a function $\mathsf{m}_k$. Considering a uniformly chosen edge, together with these marks, we obtain a sequence of random edge-rooted graphs with marks $(G_k, E_k, \mathsf{m}_k)_{k}$. We overload the notation $(\mu_k)_k$ to also refer to the associated respective measures in these marked graphs. The following lemmas will be essential in the proof of the existence of the Benjamini-Schramm limit of this object.
Recall that, via the addresses $\Addr_k$, $E_k$ corresponds to a uniform word $W_k \in \Edge^{k}_{1}$ and so the same holds true for $\mathsf{m}_{k}(E_{k}) \in \mathcal{E}_{1}^{(-k,-k+1, \dots, 1)}$.

The following lemma presents recursive relations for the degrees of the vertices $a_{k}$ and $b_{k}$ of $G_{k}$.
\begin{lemma}\label{l:degree_ak_scaling}
  We have, for every $k \in \N$
  \begin{equation}\label{eq:scaling}
    \begin{pmatrix}
      \deg^{\mathrm{out}}_{k+1}(a_{k+1}) \\
      \deg^{\mathrm{in}}_{k+1}(a_{k+1})
    \end{pmatrix}
    =
    \begin{bmatrix}
      \deg^{\mathrm{out}}_{1}(a_{1}) & \deg^{\mathrm{out}}_{1}(b_{1}) \\
      \deg^{\mathrm{in}}_{1}(a_{1}) & \deg^{\mathrm{in}}_{1}(b_{1})
    \end{bmatrix}
    \begin{pmatrix}
      \deg^{\mathrm{out}}_{k}(a_{k}) \\
      \deg^{\mathrm{in}}_{k}(a_{k})
    \end{pmatrix}.
\end{equation}
  Furthermore, for every $k \in \N$,
  \begin{equation}\label{eq:scaling_degree}
    \deg_{k}(a_{k}) \leq \big( |\Edge_{1}| - 2 \big)^{k}.
  \end{equation}
  The same relations hold for the degrees of $b_{k}$.
\end{lemma}

\begin{proof}
  Let us start with a verification of~\eqref{eq:scaling}.
  Notice that, in the construction of $G_{k+1}$, every oriented edge of $G_{k}$ is substituted by a copy of $G_{1}$, following the correct orientation.
  This implies that the edges going out from $a_{k+1}$ in $G_{k+1}$ are of two types: they either originate from an edge going out of $a_{k}$ in $G_{k}$ and an edge going out of $a_{1}$ in $G_{1}$, or they originate from an edge going into $a_{k}$ in $G_{k}$ and an edge going out of $b_{1}$ in $G_{1}$.
  This gives the relation
  \begin{equation}
  \deg^{\mathrm{out}}_{k+1}(a_{k+1}) = \deg^{\mathrm{out}}_{1}(a_{1}) \deg^{\mathrm{out}}_{k}(a_{k}) + \deg^{\mathrm{out}}_{1}(b_{1}) \deg^{\mathrm{in}}_{k}(a_{k}).
  \end{equation}
  This implies the first relation in~\eqref{eq:scaling}.
  The verification of the recursive relation for the in-degrees are analogous.

  Let us now prove~\eqref{eq:scaling_degree}.
  The proof of this fact follows an inductive argument, the base case being $\deg_{1}(a_{1}) \leq |\Edge_{1}| - 2$, since the distance between $a_1$ and $b_1$ is at least two, and $\deg_{1}(b_{1}) \geq \cut(G_{1}) \geq 2$.
  The same reasoning implies $\deg_{1}(b_{1}) \leq |\Edge_{1}| - 2$.
  Suppose~\eqref{eq:scaling_degree} holds for $k \in \N$ and notice that~\eqref{eq:scaling} implies
  \begin{equation}
    \begin{split}
      \deg_{k+1}(a_{k+1}) & = \deg^{\mathrm{out}}_{k+1}(a_{k+1}) + \deg^{\mathrm{in}}_{k+1}(a_{k+1}) \\
      & = \deg_{1}(a_{1}) \deg^{\mathrm{out}}_{k}(a_{k}) + \deg_{1}(b_{1}) \deg^{\mathrm{in}}_{k}(a_{k}) \\
      & \leq \big( |\Edge_{1}| - 2 \big) \deg_{k}(a_{k}) \\
      & \leq \big( |\Edge_{1}| - 2 \big)^{k+1}.
    \end{split}
  \end{equation}
  An analogous calculation also yields the same bound for $\deg_{k+1}(b_{k+1})$ and concludes the proof.
\end{proof}

The next lemma states that a uniformly selected edge in $G_{k}$ is typically far away from $a_{k}$ and $b_{k}$.
\begin{lemma}
  \label{l:dist_from_ek_to_ak}
  Assume $G$ satisfies condition~\eqref{e:hypothesis}. Then, for every $r > 0$, defining
  \begin{equation}
    \label{e:dist_from_ek_to_ak_1}
    k_0(r) = \frac{\log r}{\log \d_1(a_1, b_1)},
  \end{equation}
  we have
  \begin{equation}
    \label{e:dist_from_ek_to_ak_2}
    \mu_k
      \big(
        \min \{ \d_{k}(E_k, a_k), \d_{k}(E_k, b_k) \} < r
      \big)
        \leq
        2
        \Bigg(
          \frac{|\Edge_{1}| - 2}{|\Edge_{1}|}
        \Bigg)^{k - k_0(r)}
  \end{equation}
\end{lemma}

\begin{proof}
  A minimal path connecting $a_k$ to $b_k$ in $G_k$ is obtained by recursively concatenating minimal paths in each copy of $G_{k-1}$ used in the construction. Consequently, one has
  \begin{equation}
    \label{e:dist_ak_bk}
    \d_{k}(a_k,b_k)= \d_{1}(a_1,b_1)^k.
  \end{equation}
  By the above, we know that for $k_{0} = k_0(r)$ defined as in \Cref{e:dist_from_ek_to_ak_1}, we have $\d_{k_0}(a_{k_0}, b_{k_0}) \geq r$. Then, for $k > k_0$, in order for $E_k$ to be at distance at most $r$ from $a_k$, $E_k$ must be contained inside a graph of the collection $\{ G^{w}_k; \, w \in \Edge_{1}^{k - k_0} \}$ incident to $a_k$ (that is, a graph of the $k_0$-th scale contained in $G_k$, see Equation~\eqref{eq:decompose}).
  The same is true for $b_k$. The number of such graphs is then
  \begin{equation}
    \label{e:dist_from_ek_to_ak_3}
      \deg_{k-k_0}(a_{k - k_0}) + \deg_{k-k_0}(b_{k - k_0}).
  \end{equation}
  At the same time, the total number of subgraphs of scale $k_{0}$ contained in $G_k$ is $|\Edge_{1}|^{k - k_0}$.
  Therefore, by \cref{l:degree_ak_scaling}, the probability of a uniformly chosen edge being in a graph of the collection incident to either $a_k$ or $b_k$ is upper bounded by
  \begin{equation}
    \label{e:dist_from_ek_to_ak_4}
      \frac{
        \deg_{k-k_0}(a_{k - k_0})
        +
        \deg_{k-k_0}(b_{k - k_0})
      }{
        |\Edge_{1}|^{k - k_0}
      }
      \leq
      2
      \Bigg(
        \frac{|\Edge_{1}| - 2}{|\Edge_{1}|}
      \Bigg)^{k - k_0},
  \end{equation}
  proving the result.
\end{proof}

In what follows, for $j \leq k$ and $\omega \in \Edge_{1}^{k-j}$, we let $a_{k}^w$ and $b_{k}^w$ denote the vertices of $G_{k}^w$ corresponding to $a_{j}$ and $b_{j}$, respectively.

\begin{lemma}
  \label{l:tight_Gk_ek}
  The family $(\deg(E_k))_k$ is tight.
  Therefore, $(\bar{G}_k)_k \in \mathcal{G}$ and $(G_k, E_k, \mathsf{m}_k)_k \in \mathcal{G}_\mathcal{A}$ are tight sequences in $\mathcal{G}$ and $\mathcal{G}_{\mathcal{A}}$, respectively.
\end{lemma}

\begin{proof}
  Given $\eps > 0$, we use \Cref{l:dist_from_ek_to_ak} to find $k_1$ large enough such that with probability $1 - \eps$, $E_{k_1}$ is not incident to either $a_{k_1}$ or $b_{k_1}$. Let $M$ be the maximum degree of $G_{k_1}$. We sample $E_k$ in $G_k$ in the following manner: we first sample a word $w \in \Edge_{1}^{k-k_1}$ uniformly, thus obtaining a graph of the form $G^{w}_k$ of the $k_1$-th scale contained in $G_k$, and then sample a word $w' \in \Edge_{1}^{k_1}$ again uniformly, thus obtaining a uniform word in $\Edge_{1}^k$, and equivalently, a uniform edge of $\Edge_k$. Note that the only way for $\deg(E_k) > 2M$ is if in the second step above, the uniformly chosen edge within $G_{k}^w$ is adjacent to either endvertex $a_k^w$ or $b_k^w$ of this graph of the $k_1$-th scale. But this has probability at most $\eps$, by construction, finishing the proof of the first statement. Tightness for $(\bar{G}_k)_k \in \mathcal{G}$ and $(G_k, E_k, \mathsf{m}_k)_k \in \mathcal{G}_\mathcal{A}$ now follows directly from Proposition~\ref{p:tightness}.
\end{proof}

\begin{lemma}
  \label{l:cauchy_Gk}
  Given a fixed connected edge-rooted graph $(H, e)$, under condition \eqref{e:hypothesis},
  \begin{equation}
    \label{e:cauchy_Gk_1}
    \mu_k
      \big(
        B_{k}(E_k, r) = (H, e)
      \big)
  \end{equation}
  is a Cauchy sequence.
  
  Furthermore, given a connected edge-rooted marked graph $(H,e,\mathsf{m}) \in \mathcal{G}_{\mathcal{A}}$, and $\varepsilon>0$,
  \begin{equation}
    \label{e:cauchy_Gk_2}
    \mu_k
      \big(
        \d( (G_{k}, E_{k}, \mathsf{m}_{k} ) , (H, e, \mathsf{m})) \leq \varepsilon
      \big)
  \end{equation}
  is a Cauchy sequence.
\end{lemma}

\begin{proof}
  Let $r$ be the diameter of $H$, that is the maximum distance between its vertices. Let $A_{k}$ be the event where $E_{k}$ is at distance larger than $r$ from the endpoints $a_{k}$ and $b_{k}$. Using Lemma~\ref{l:dist_from_ek_to_ak}, given $\eps > 0$, we can find $k_0$ large enough such that
  \begin{equation}
    \label{e:cauchy_Gk_3}
    \mu_{k_0} \big( A_{k_0} \big) \geq 1 - \eps.
  \end{equation}
  For $k > k_0$, we sample $E_k$ in $G_k$ in two steps as in the previous lemma.
  First uniformly sample a word $\omega \in \mathcal{E}_{1}^{k - k_0}$ and look at the graph $G^{\omega}_{k}$.
  Then select a word $\omega' \in \mathcal{E}_{1}^{k_{0}}$, which yields a uniformly selected edge $G^{\omega}_{k}$, and thus a uniformly selected edge in $G_{k}$.
  Let $A_{k_0, k}$ be the event where $E_k$ is at distance larger than $r$ from the endpoints $a^{\omega}_{k}$ and $b^{\omega}_{k}$ of the graph $G^{\omega}_{k}$.
  By definition of the measures, we have
  \begin{equation}
    \label{e:cauchy_Gk_4}
    \begin{split}
      \mu_{k}
        \big(
          B_{k}(E_k, r) = (H, e), A_{k_{0},k}
        \big)
      &=
      \mu_{k_0}
        \big(
          B_{k_0}(E_{k_0}, r) = (H, e), A_{k_0}
        \big)
      \\
      \mu_{k}
        \big(
          A_{k_{0}, k}
        \big)
      &=
      \mu_{k_0}
        \big(
          A_{k_0}
        \big).
    \end{split}
  \end{equation}
  The relations above immediately implies that, for $k, k' > k_0$,
  \begin{equation}
    \label{e:cauchy_Gk_5}
    \Big|
      \mu_k
      \big(
        B_{k}(E_k, r) = (H, e)
      \big)
      -
      \mu_{k'}
      \big(
        B_{k'}(E_{k'}, r) = (H, e)
      \big)
    \Big|
      <
        2\eps,
  \end{equation}
  concluding the proof of the first statement. The second follows the same lines and we omit it here.
\end{proof}

The lemmas above then imply the existence of the Benjamini Schramm limit of the sequence $(G_k, E_k)$. We summarize this result in what follows.

\begin{proposition}
  \label{p:benjamini_schramm}
  The sequence $\bar{G}_k = (G_{k}, E_{k})$ converges in the sense of \Cref{d:besh} to $\mu_\infty$, the law of a random edge-rooted graph $\bar{G}_\infty = (G_\infty, E_\infty) \in \mathcal{G}$. An analogous statement holds for marked graphs: the sequence $(G_{k}, E_{k}, \mathsf{m}_{k})$ converges in the sense of \Cref{d:besh} to a random edge-rooted marked graph $(G_\infty, E_\infty, \mathsf{m}_{\infty}) \in \mathcal{G}_{\mathcal{A}}$.
\end{proposition}

\begin{proof}
Tightness follows from Lemma~\ref{l:tight_Gk_ek}. By Lemma~\ref{l:cauchy_Gk}, every convergent subsequence of $(G_k, E_k)_k$ must have the same infinite limit, proving the first result. The case of marked graphs is analogous.
\end{proof}

Having established the existence of the Benjamini-Schramm limits, we now turn our attention to understanding the local neighborhoods of $E_\infty$ in $G_\infty$ in terms of graphs of the form $G_k$. The main result we need is one allowing us to tile $G_\infty$ with edge disjoint copies of $G_k$ for every $k$, this tiling being compatible for different values of $k$.

We consider the right-shift $\sigma : \Edge_{1}^{\Z_-} \to \Edge_{1}^{\Z_-}$ over infinite words.

\begin{proposition}
  \label{p:benjamini_schramm_marks}
  The distribution $\bar{G}_\infty = (G_\infty, E_\infty) \in \mathcal{G}$ has the following properties:
  \begin{enumerate}[\quad a)]
    \item The address of the edge root $E_\infty$ is an infinite i.i.d.\ sequence of random variables $(e_j)_{j \in \Z_-}$, each $e_j$ being chosen uniformly in $\Edge_1$.
    \item For each $k$, $G_\infty$ is can be tiled by copies of $G_k$. That is, $G_\infty$ is an edge-disjoint union of subgraphs $G_k^w$ isomorphic to $G_k$, where $w \in \sigma^k(\Edge_{1}^{\Z_-})$. Furthermore, tilings for different levels are compatible, in the sense that, for $k > k'$, $G_k^w$ is tiled by graphs of the form $G_{k'}^{w'}$, with $w' \in \sigma^{k'}(\Edge_{1}^{\Z_-})$.
    \item $E_\infty$ is contained in a nested sequence of graphs
    \begin{equation}
      G_1^{w_1} \subset G_2^{w_2} \subset \dots \subset G_k^{w_k} \subset \dots ,
    \end{equation}
    each $G_k^{w_k}$ being an element of a tiling in the above item, moreover $w_k = \sigma^{k}(\dots, e_{-1}, e_0) = (\dots, e_{-(k + 1)}, e_{-k})$.
  \end{enumerate}
\end{proposition}

\begin{proof}
As established in Proposition~\ref{p:benjamini_schramm}, the sequence of marked graphs $(G_{k}, E_{k}, \mathsf{m}_{k})_{k}$ converges in the Benjamini-Schramm sense to the marked graph $(G_{\infty}, E_{\infty}, \mathsf{m})$.

Let us first argue that the addresses for different edges in $(G_{\infty}, E_{\infty}, \mathsf{m})$ are different almost surely.

Recall, for $r \geq 1$, the restriction map $\omega \mapsto \omega|_{r}$ defined in~\eqref{e:address_restriction_def2} and notice that it is continuous and extends to a continuous map defined in $\mathcal{G}_{\mathcal{A}}$.

For $r \geq 1$, define the open set
\begin{equation}
\mathcal{O}_{r} = \bigg\{
  \begin{array}{c}
    (G, E, \mathsf{m}) \in \mathcal{G}_{\mathcal{A}} : \text{ there exist distinct edges } e_{1}, e_{2} \in G \\
    \text{ such that } \d_{G} (e_{1}, E) \leq r, \d_{G}(e_{2}, E) \leq r, \text{ and } \mathsf{m}(e_{1})|_{2r} = \mathsf{m}(e_{2})|_{2r}
  \end{array}
\bigg\},
\end{equation}
and notice that Portmanteu's Theorem yields
\begin{equation}
\mu_{\infty} \big( (G_{\infty}, E_{\infty}, \mathsf{m}) \in \mathcal{O}_{r} \big) \leq \liminf \mu_{k} \big( (G_{k}, E_{k}, \mathsf{m}_{k})_{k} \in \mathcal{O}_{r} \big).
\end{equation}
Below we argue that the last probability equals zero, due to the fact that, in $(G_{k}, E_{k}, \mathsf{m}_{k})_{k}$, all addresses are uniquely determined. In fact, if two distinct edges in $(G_{k}, E_{k}, \mathsf{m}_{k})$ are at distance at most $2r$ from each other, the restriction of their address to the first $2r$ entries must differ. Taking the limit as $r$ grows implies that any two distinct edges in $(G_{\infty}, E_{\infty}, \mathsf{m})$ have different addresses.

Let us now turn our attention to the proof of~\textit{a}). For any fixed $r \geq 1$, the Continuous Mapping Theorem implies that the sequence $(G_{k}, E_{k}, \mathsf{m}_{k}|_{r})_{k}$ converges to the marked graph $(G_{\infty}, E_{\infty}, \mathsf{m}|_{r})$. Notice that the finite-size address $\mathsf{m}_{k}(E_{k})_{r}$ form an i.i.d.\ sequence taken uniformly from $\Edge_{1}^{r}$. From the Benjamini-Schramm convergence, we obtain that $\mathsf{m}(E_{\infty})|_{r}$ is distributed uniformly in $\Edge_{1}^{r}$, implying the first statement.

In order to define the tilling, given $\omega \in \sigma^{k}(\Edge_{1}^{\Z_{-}})$, we define $G_{k}^{\omega}$ as the subgraph induced by the collection of edges $e \in G_{\infty}$ such that $\sigma^{k} \circ \mathsf{m}(e) = \omega$. Immediately from the definition we obtain that the graphs $G_{k}^{\omega}$ and $G_{k}^{\omega'}$ are edge-disjoint for $\omega \neq \omega'$ (here we use that the address is almost surely unique) and that the tillings of $G_{\infty}$ is consistent for different values of $k$.

The only thing left to verify is that $G_{k}^{\omega}$ is isomorphic to $G_{k}$. This follows from and inductive argument based on the fact that the distribution of $(G_{\infty}, E_{k}, \mathsf{m}_{\infty})$ is preserved by the operation of substituting each edge of $G_{\infty}$ by a copy of $G_{1}$ (as in the construction of the graphs $G_{k}$), selecting a root uniformly in the graph corresponding to the edge $E_{\infty}$ and updating the marks accordingly. This invariance is proved by observing that the same holds true for the sequence $(G_{k}, E_{k}, \mathsf{m}_{k})$, where this operation gives origin to the random graph $(G_{k+1}, E_{k+1}, \mathsf{m}_{k+1})$.

Item \textit{c} in the statement follows trivially from item \textit{b} and concludes the proof.
\end{proof}

\section{Percolation on the Benjamini-Schramm limit}
\label{sec:bensch_perc}

In Section~\ref{sec:notation}, we introduced Bernoulli percolation on the graphs $G_{k}$ (see Equation~\eqref{eq:percolation_k}). We start this section by extending this notion in an analogous way to the limiting random graph $\bar{G}_{\infty} = (V_{\infty}, \Edge_{\infty})$.

Consider an independent family of uniform random variables $(U_{e})_{e \in \Edge_{\infty}}$ and assume that this family is also independent of the realization of $\bar{G}_{\infty}$. Given $p \in [0,1]$, define the percolation configuration as
\begin{equation}\label{eq:configuration_coupling}
\omega_{p}(e) = \textbf{1}_{U_{e} \leq p}, \quad \text{for } e \in \Edge_{\infty}.
\end{equation}
Fixed a realization of the graph $\bar{G}_{\infty}$, let $\PP^{\bar{G}_{\infty}}_{p}$ denote the distribution of the random configuration $\omega_{p}$ on $\{0,1\}^{\Edge_{\infty}}$.

\begin{remark}
Notice one can also construct the random graph $\bar{G}_{\infty}$ together with the configuration $\omega_{p}$ as the Benjamini-Schramm limit of the edge-marked graphs $(G_{k}, \omega_{p})$, where $\omega_{p}$ in defined in $G_{k}$ as in~\eqref{eq:percolation_k}.
\end{remark}

As usual, for any two sets $A, B \subset V_{\infty}$, we denote by $[A \longleftrightarrow B]$ the event that a vertex $a \in A$ is connected to a vertex $b \in B$ by an open path. We also write
\begin{equation}
[A \longleftrightarrow \infty] = \bigcup_{a \in A} \bigcap_{r \geq 1} \big[ \{a\} \longleftrightarrow B_{\infty}(a, r)^{c} \big],
\end{equation}
for the event where there exists an infinite open path starting from some vertex in $A$.

Given a realization $\bar{G}_\infty = (G_\infty, E_\infty)$ and $p \in [0,1]$, define the random variable $\theta^{\bar{G}_\infty}(p) = \PP^{\bar{G}_\infty}_p[ E_\infty \leftrightarrow \infty ]$ (note that this function depends on the realization of the graph $G_\infty$ as well as on the root edge $E_{\infty}$). Define also the (in principle random) percolation threshold
\begin{equation}
  \label{eq:p_c}
  p_c(\bar{G}_\infty) = \sup \big\{ p \in [0, 1]: \theta^{\bar{G}_\infty}(p) = 0 \big\}.
\end{equation}
The next result rules out the discomforting possibility that $p_c$ is random.
\begin{theorem}
  \label{t:p_c_constant}
  Under~\eqref{e:hypothesis}, the critical value $p_c(\bar{G}_\infty)$ is a.s. constant under the law of $\bar{G}_\infty$ and it coincides with the box-crossing threshold $p_\star = p_{\star}(G_{1})$.
\end{theorem}

\begin{proof}
  It is easy to see that $p_c$ does not depend on the edge $E_\infty$ by Harris-FKG's inequality, but we now prove that it also does not depend on the randomness given by $G_\infty$.
  We start by showing that $p_c(\bar{G}_\infty) \geq p_\star$ and for this we fix $p < p_\star$.

  Recall the tilling $\{G^{w}_{\ell}: w \in \sigma^{\ell}(\Edge_{1}^{\Z_{-}})\}$ of $G_{\infty}$ from Proposition~\ref{p:benjamini_schramm_marks}, and denote by $\tilde{\omega}$ the unique word associated with the root edge $E_\infty$.
  Notice that
  \begin{equation}
    [E_\infty \leftrightarrow \infty] \subseteq
    \Big[ \text{one of $G_\ell^{w}$ neighboring $G_\ell^{\sigma^{\ell}(\tilde{w})}$ is crossed} \Big],
  \end{equation}
  so that
  \begin{equation}
    \begin{split}
      \PP_{p}^{\bar{G}_\infty} (E_\infty \leftrightarrow \infty)
      & \leq \# \{ G_\ell^{w} \text{ neighboring } G_\ell^{\sigma^{\ell}(\tilde{w})} \}
      \mathbb{P}^{\ell}_{p} (G_\ell \text{ is crossed}) \\
      & \leq \# \{ G_\ell^{w} \text{ neighboring } G_\ell^{\sigma^{\ell}(\tilde{w})} \}
      c(p) \exp \Big\{ -c'(p) \d_{1}(a_1, b_1)^\ell \Big\},
    \end{split}
  \end{equation}
  where the last estimate follows directly from Theorem~\ref{th:off-critical-decay}.
  
  Therefore, according to Lemma~\ref{l:tight_Gk_ek}, the (non-negative) random variable $\theta^{\bar{G}_\infty}$ is bounded by a product between a tight family of random variables and a sequence that goes to zero, and thus it must equal zero almost surely.
  This proves that $p_c(\bar{G}_\infty) \geq p_\star$.

  We now turn to the complementary inequality and fix $p > p_\star$.
  Observe that,
  \begin{equation}
    \label{eq:l_argument}
    \bigcap_{\ell \geq 1} \bigcap_{G_\ell^w \subseteq G_{\ell + 1}^{\sigma^{\ell+1}(\tilde{\omega})}}
    \big[ G_\ell^w \text{ is crossed} \big]
    \subseteq [E_\infty \leftrightarrow \infty ].
  \end{equation}
  Therefore, by \cref{eq:off-critical-decay-super} and Harris-FKG's inequality,
   \begin{equation}
    \label{eq:lower_bound_percolate}
    \PP_{p}^{\bar{G}_\infty}(E_\infty \leftrightarrow \infty)
    \geq p^{|\Edge_{\ell_{0}}|} \prod_{\ell \geq \ell_0 + 1} \Big(
    1 - c(p) \exp\{ -c'(p) \cut(G_{1})^\ell \}
    \Big)^{|\mathcal{E}|} > 0.
   \end{equation}
  Therefore, the random variable $\theta^{\bar{G}_\infty}$ is uniformly bounded from below as we vary the marked graph $\bar{G}_\infty$.
  In particular, $p_c(\bar{G}_{\infty}) \leq p_\star$, finishing the proof of the lemma.
\end{proof}

\subsection{Uniqueness and continuity}\label{sec:uniq_cont}

Given a realization $\bar{G}_{\infty} = (G_{\infty}, E_{\infty})$ and $\big(\omega_{p}(e)\big)_{e \in G_{\infty}}$, denote by $N$ the number of infinite open connected components.
\begin{theorem}[Uniqueness of the infinite cluster]\label{t:uniqueness}
For all $p > p_{c}$,
\begin{equation}
\PP_{p}^{\bar{G}_{\infty}} \big( N = 1 \big) = 1, \quad \bar{G}_{\infty}-a.s.
\end{equation}
\end{theorem}

\begin{remark}
Contrary to the Burton-Keane finite-energy argument, we here prove that any two infinite clusters intersect with probability one without resorting to local changes of the configuration, but actually relying on an application of Borel-Cantelli Lemma to ensure the existence of crossings of graphs in different scales. In particular, this surpasses the geometric argument based on trifurcation points used to rule out the possibility that infinitely many disjoint infinite open clusters happen at once.
\end{remark}

\begin{proof}[Proof of Theorem~\ref{t:uniqueness}]
Let us prove that any two infinite clusters need to intersect.
Assume $\mathcal{C}$ and $\tilde{\mathcal{C}}$ are infinite open clusters and fix $\ell_{0}$ large enough so that some graph $G^{w}_{\ell_{0}}$ of  the level-$\ell_{0}$ tilling of $\bar{G}_{\infty}$ that intersects both components.

For $\ell > \ell_{0}$, $G^{\sigma^{\ell - \ell_{0}}(w)}_{\ell}$ is the graph of the level-$\ell$ tilling that contains $G^{w}_{\ell_{0}}$.
Since there are at least two disjoint paths connecting $a_{1}$ to $b_{1}$ in $G_{1}$, it follows that
\begin{equation}
\PP_{p}^{\bar{G}_{\infty}} \big( a_{\ell} \leftrightarrow b_{\ell} \text{ in } G^{\sigma^{\ell - \ell_{0}}(w)}_{\ell} \setminus G^{\sigma^{(\ell - 1) - \ell_{0}}(w)}_{\ell-1} \big) \geq p_{\ell-1}^{ \diam(G_{1}) } \geq p_{c}^{\diam(G_{1})}.
\end{equation}
Borel-Cantelli Lemma implies that infinitely many crossings as described in the event above happen almost surely.
In order to conclude the proof, simply notice that, since both $\mathcal{C}$ and $\tilde{\mathcal{C}}$ intersect $G^{w}_{\ell_{0}}$, they must cross either $a_{\ell}$ or $b_{\ell}$ in the higher level pavings.
As soon as these two vertices are connected, the components are necessarily merged.
This concludes the proof.
\end{proof}

\begin{remark}\label{rmk:theta_p_c}
  Let us present a simple argument that implies $\theta^{\bar{G}_{\infty}}(p_{c}) = 0$ almost surely.
  Notice that, almost surely on the realization of $\bar{G}_{\infty}$, there are infinitely many values of $\ell$ such that $G^{\sigma^{2\ell+1}(E_{\infty})}_{2\ell+1}$ is not adjacent to either $a_{2\ell+3}$ or $b_{2\ell+3}$.
  Denote by $A_{\ell}$ and $B_{\ell}$ the (random) sequence of these extreme vertices where the above holds (notice this sequence depends only on the realization of $\bar{G}_{\infty}$).
  Define now the events
  \begin{equation}
    D_{2\ell+3} = \bigg\{ \begin{array}{c}
      \text{all graphs } G^{w}_{2\ell+1} \subset G^{\sigma^{2\ell+3}(E_\infty)}_{2\ell+3} \text{ adjecent to either} \\
      a_{2\ell+3} \text{ or } b_{2\ell+3} \text{ are not crossed}
    \end{array} \bigg\}.
  \end{equation}
  Notice these events are independent and, since $p_{c} = p_{\star}$,
  \begin{equation}
    \PP_{p_{c}}^{\bar{G}_{\infty}}\big( D_{2\ell+3} \big) \geq \big(1-p_{\star}\big)^{\deg(a_2) + \deg(b_2)}.
  \end{equation}

  Finally, observe that
  \begin{equation}
    \PP_{p_{c}}^{\bar{G}_{\infty}} \big( E_{\infty} \notleftright \{A_{\ell}, B_{\ell}\} \big) \geq \PP_{p_{c}}^{\bar{G}_{\infty}} \big(D_{2\ell+3} \big).
  \end{equation}
  The proof now follows from an application of Borel-Cantelli Lemma.
  Another proof of this fact will be obtained in Section~\ref{s:arm_exponents} as a consequence of the existence of arm exponents.
\end{remark}

The following theorem regards the continuity of the percolation function $\theta^{\bar{G}_{\infty}}$ and its proof is analogous to the case of percolation on the Euclidean lattice.
\begin{theorem}[Continuity of the percolation function]
  \label{t:continuity}
  The function $p \mapsto \theta^{\bar{G}_\infty}(p)$ is almost surely continuous.
\end{theorem}

\begin{proof}
  Observe first that $\theta^{\bar{G}_{\infty}}$ is clearly right-continuous, as it is the pointwise limit of nondecreasing continuous functions
  \begin{equation}
    \theta^{\bar{G}_{\infty}}(p) = \lim_{n} \PP_{p}^{\bar{G}_{\infty}} \big( E_{\infty} \leftrightarrow \partial B (E_{\infty}, n) \big).
  \end{equation}

  It remains to prove the left continuity of the function $\theta^{\bar{G}_{\infty}}$.
  This is immediate for $p \in [0, p_{c}]$, since $\theta^{\bar{G}_{\infty}}(p_{c}) = 0$ due to Remark~\ref{rmk:theta_p_c}.
  Assume that $\theta^{\bar{G}_{\infty}}$ is not left-continuous on a point $p' > p_{c}$.

  We now make use of the coupling between configurations at different densities introduced in~\eqref{eq:configuration_coupling}.
  Let us denote by $\mathcal{C}_{p}$ the unique infinite component of the configuration at level $p$.
  Since $\theta^{\bar{G}_{\infty}}$ has a jump at $p'$, there is a positive probability that $E_{\infty} \notin \mathcal{C}_{p}$, for all $p < p'$ and $E_{\infty} \in \mathcal{C}_{p'}$.
  In particular, fixed $\tilde{p} \in (p_{c}, p')$, there exists a path of edges $e_{1}, e_{2}, \dots, e_{n}$ that connects $E_{\infty}$ to $\mathcal{C}_{\tilde{p}}$ that is open at level $p'$ but not entirely open at any level below $p'$.
  This implies that at least one of the edges in the path necessarily satisfies $U_{e_{\ell}} = p'$, where $U$ denotes the uniform random variable used in the coupling~\eqref{eq:configuration_coupling}.
  Since this probability is zero, the function $\theta^{\bar{G}_{\infty}}$ is necessarily continuous.
  This concludes the proof.
\end{proof}

\section{Arm exponents and cluster dimension}
\label{s:arm_exponents}

For $k \geq 0$ we denote by $\mathcal{C}_{\{a_k, b_k\}}$ the cluster boundary of $G_k$, that is
\begin{equation}
  \label{eq:boundary_cluster}
  \mathcal{C}_{\{a_k, b_k\}} := \Big\{
  e \in \Edge_k : e \longleftrightarrow \{a_k, b_k\}
  \Big\}
\end{equation}

Recall $\mathbb{P}_{p}^{k}$ denotes the probability induced in $G_{k}$ when edges are open with probability $p \in [0,1]$. In this section we establish the following result about $\mathcal{C}_{\{a_k, b_k\}}$.
As discussed below, a version of this was established earlier in \cite{hk2009}.

\nc{c:volume_low}
\nc{c:volume_high}
\begin{theorem}
  \label{th:volume-of-critical-cluster}
  There exists constants $d_f$ (referred as the fractal dimension of $\mathcal{C}_{\{a_k, b_k\}}$) and $\uc{c:volume_low}, \uc{c:volume_high} > 0$ such that, at $p_\star$,
  \begin{equation}
    \label{eq:expected_volume_critical}
    \uc{c:volume_low} d_f^k
    \leq \mathbb{E}_{p_\star}^{k} \big[ |\mathcal{C}_{\{a_k, b_k\}}| \big]
    \leq \uc{c:volume_high} d_f^k.
  \end{equation}
  Moreover, $d_f$ satisfies:
  \begin{enumerate}[\quad a)]
  \item $d_f < |\Edge_{1}|$,
  \item $d_f > \max \{ \deg^{\mathrm{out}}(a_1), \deg^{\mathrm{in}}(b_1), \deg(a_1) \wedge \deg(b_1) \}$.
  \end{enumerate}
  As a consequence, we obtain an estimate for the annealed one-arm connection probability.
  If $E_k$ denotes an uniformly selected edge in $G_k$, then
  \begin{equation}
    \label{eq:one_arm_exponent}
    \uc{c:volume_low} \d_{k}(a_k, b_k)^{-\alpha_1}
    \leq \; \mathbb{P}_{p_\star}^{k} \big( E_k \leftrightarrow \{a_k, b_k\} \big)
    \leq \; \uc{c:volume_high} \d_{k}(a_k, b_k)^{-\alpha_1},
  \end{equation}
  where $\alpha_1 = \log \big( |\Edge_{1}|/d_f \big) / \log \d_{1}(a_1, b_1)$ is called the one-arm exponent.
\end{theorem}

\begin{remark}\label{rmk:one_arm}
The statement in Equation~\eqref{eq:one_arm_exponent} can be seen as determining the one-arm exponent in hierarchical percolation, akin to that established for critical site percolation on the triangular lattice by Lawler, Schramm, and Werner~\cite{lsw_2002}.

We remark that some of the statements above were first proved by Hambly and Kumagai~\cite{hk2009} for the case of the diamond hierarchical lattice. The strategy of the proof used here is similar to the one from~\cite{hk2009}, but we are able to cover a larger class of generating graphs, in part due to Theorem~\ref{t:crossing}.
Furthermore, some of the intermediary results obtained during the proof below will be used in \cref{s:exponent_relations} to analyze the scaling relations.
In particular, the off-critical dynamical system~\eqref{eq:x_ab_evolution_2}, closely related to~\eqref{eq:x_ab_evolution}, is central to the work carried out in Section~\ref{s:exponent_relations}.
\end{remark}

\begin{remark}
The fractal dimension $d_{f}$ can be obtained as the largest eigenvalue of the matrix $M(p_{\star})$ defined in~\eqref{eq:M_p}. This provides an exact expression for $d_{f}$ that we exploit in order to obtain the bounds in Theorem~\ref{th:volume-of-critical-cluster}.
\end{remark}

As in the above, we here explore the recursive behavior of the graphs $G_{k}$ to understand the quantity $\mathbb{E}_{p_\star}^{k} \big[ |\mathcal{C}_{\{a_k, b_k\}}| \big]$.
However, it is necessary to keep track of more information at each scale.
More precisely, let
\begin{equation}
  \label{eq:three_expectations}
  \begin{split}
    x_k^a(p) & = \mathbb{E}^k_p \Big[ \# \big\{
    e \in \Edge_k; e \leftrightarrow a_k \text{ but } e \notleftright b_k
    \big\} \Big],\\
    x_k^b(p) & = \mathbb{E}^k_p \Big[ \# \big\{
    e \in \Edge_k; e \notleftright a_k \text{ but } e \leftrightarrow b_k
    \big\} \Big], \text{ and}\\
    x_k^{ab}(p) & = \mathbb{E}^k_p \Big[ \# \big\{
    e \in \Edge_k; e \leftrightarrow a_k \text{ and } e \leftrightarrow b_k
    \big\} \Big].
  \end{split}
\end{equation}
Note that $x_0^a(p) = x_0^b(p) = 0$ and $x_0^{ab}(p) = 1$ for every $p \in [0, 1]$.
In what follows, we examine how these quantities change as we increase the scale $k$.

Fix $k \geq 0$ and recall that we can tile the graph $G_{k + 1}$ by copies of $G_{k}$ denoted by $\{G_{k+1}^{e'}\}_{e' \in \Edge_{1}}$ (see Equation~\eqref{eq:decompose}), so for example
\begin{equation}
  \label{eq:decompose_volume}
  x_{k + 1}^a = \sum_{e' \in \Edge_{1}}\mathbb{E}^{k + 1}_p \Big[ \# \big\{
  e \text{ edge of } G_{k+1}^{e'}; e \leftrightarrow a_{k + 1}
  \text{ but } e \notleftright b_{k + 1}
  \big\} \Big]
\end{equation}
and similarly for $x_{k + 1}^b$ and $x_{k + 1}^{ab}$.
For $e' \in \Edge_{1}$, we re-write its corresponding contribution in~\cref{eq:decompose_volume} in terms of $x_k^\cdot$'s, but for this we will need to eliminate $G_{k+1}^{e'}$ from $G_{k + 1}$.
Write $\mathbb{P}_p^{k + 1, e'}$ for the probability measure induced on the graph with edges in $\{G_{k+1}^{e''}\}_{e'' \in \Edge_{1} \setminus \{e'\}}$. We also denote by $a_{k}^{e'}$ (respectively, $b_{k}^{e'}$) for the vertex corresponding to $a_{k}$ (respectively, $b_{k}$) in the copy $G^{e'}_{k+1}$ associated to the edge $e' \in \Edge_{1}$. We obtain
\begin{equation}
  \label{eq:volume_e_prime}
  \begin{split}
    \mathbb{E}^{k}_p \Big[ \# & \big\{
    e \text{ edge of } G_{k+1}^{e'}; e \leftrightarrow a_{k + 1}
    \text{ but } e \notleftright b_{k + 1}
    \big\} \Big]\\
    & = x_k^a \; \mathbb{P}_p^{k + 1 , e'} \big(
    a_k^{e'} \leftrightarrow a_{k + 1} \text{ but } a_k^{e'} \notleftright b_{k + 1}
    \big)\\
    & \phantom{=} + x_k^b \; \mathbb{P}_p^{k + 1 , e'} \big(
    b_k^{e'} \leftrightarrow a_{k + 1} \text{ but } b_k^{e'} \notleftright b_{k + 1}
    \big)\\
    & \phantom{=} + x_k^{ab} \; \mathbb{P}_p^{k + 1 , e'} \big(
      \{a_k^{e'}, b_k^{e'}\} \leftrightarrow a_{k + 1}
      \text{ but } \{a_k^{e'}, b_k^{e'}\} \notleftright b_{k + 1}
    \big).
  \end{split}
\end{equation}
The above can be written in terms of $G_1$ as
\begin{equation}
  \label{eq:volume_e_prime2}
  \begin{split}
    & x_k^a \; \mathbb{P}_{f_{k}(p)}^{1 , e'} \big(
    e'_- \leftrightarrow a_1 \text{ but } e'_+ \notleftright b_1
    \big)\\
    & \phantom{=} + x_k^b \; \mathbb{P}_{f_{k}(p)}^{1 , e'} \big(
    e'_+ \leftrightarrow a_1 \text{ but } e'_+ \notleftright b_1
    \big)\\
    & \phantom{=} + x_k^{ab} \; \mathbb{P}_{f_{k}(p)}^{1 , e'} \big(
      e' \leftrightarrow a_1
      \text{ but } e' \notleftright b_1
    \big).
  \end{split}
\end{equation}
Therefore, if we define the matrix $M(p)$ as the element-wise sum over all $e' \in \Edge_{1}$ of
\begin{equation}
  \label{eq:M_p}
  \begin{pmatrix}
    \mathbb{P}_{p}^{1 , e'} \big(
    e'_- \leftrightarrow a_1, e'_- \notleftright b_1
    \big)
    & \mathbb{P}_{p}^{1 , e'} \big(
      e'_+ \leftrightarrow a_1, e'_+ \notleftright b_1
      \big)
    & \mathbb{P}_{p}^{1 , e'} \big(
      e' \leftrightarrow a_1, e' \notleftright b_1
      \big)\\
    \mathbb{P}_{p}^{1 , e'} \big(
    e'_- \notleftright a_1, e'_- \leftrightarrow b_1
    \big)
    & \mathbb{P}_{p}^{1 , e'} \big(
      e'_+ \notleftright a_1, e'_+ \leftrightarrow b_1
      \big)
    & \mathbb{P}_{p}^{1 , e'} \big(
      e' \notleftright a_1, e' \leftrightarrow b_1
      \big)\\
    \mathbb{P}_{p}^{1 , e'} \big(
    e'_- \leftrightarrow a_1, e'_- \leftrightarrow b_1
    \big)
    & \mathbb{P}_{p}^{1 , e'} \big(
      e'_+ \leftrightarrow a_1, e'_+ \leftrightarrow b_1
      \big)
    & \mathbb{P}_{p}^{1 , e'} \big(
      e' \leftrightarrow a_1, e' \leftrightarrow b_1
      \big)
  \end{pmatrix},
\end{equation}
we can write
\begin{equation}
  \label{eq:x_ab_evolution}
  (x_{k + 1}^a, x_{k + 1}^b, x_{k + 1}^{ab})^{\intercal}
  = M(f_{k}(p)) \cdot (x_k^a, x_k^b, x_k^{ab})^{\intercal}.
\end{equation}
This dynamical system will be crucial for the proof of Theorems~\ref{th:volume-of-critical-cluster} and~\ref{th:beta_and_relations}.

\begin{proof}[Proof of \cref{th:volume-of-critical-cluster}]
  Due to the fact that $p_\star$ is a fixed point of $f$, we can write
  \begin{equation}
    \mathbb{E}_{p_\star}^k \big[ |\mathcal{C}_{\{a_k, b_k\}}| \big]
    = x_k^a + x_k^b + x_k^{ab}
    = (1, 1, 1) \cdot M(p_\star)^k (0, 0, 1)^\intercal.
  \end{equation}
  Since the Perron-Frobenius matrix $M(p_\star)$ is irreducible, \cref{eq:expected_volume_critical} follows with $d_f$ defined as the principal eigenvalue of $M(p_\star)$ (recall that all entries in the left eigenvector are positive).

  All we need to do now is to verify that $d_f$ satisfies the stated properties.
  We start by proving that $d_f < |\Edge_1|$ as this is done by first extending the matrix $M(p_{\star})$ with one more column and one more row, to represent the possibility of having disconnected subgraphs.
  More precisely, let
  \begin{equation*}
    M' = \frac{1}{|\Edge_1|}\sum_{e' \in \Edge_{1}}
    \begin{pmatrix}
      M_{1, 1}(e')&M_{2, 1}(e')&M_{3, 1}(e')& 0\\
      M_{1, 2}(e')&M_{2, 2}(e')&M_{3, 2}(e')& 0\\
      M_{1, 3}(e')&M_{2, 3}(e')&M_{3, 3}(e')& 0\\
      \mathbb{P}^{1, e'}_{p_\star} (e'_- \notleftright \{a_1, b_1\})
      & \mathbb{P}^{1, e'}_{p_\star} (e'_+ \notleftright \{a_1, b_1\})
      & \mathbb{P}^{1, e'}_{p_\star} (e' \notleftright \{a_1, b_1\})
      & 1
    \end{pmatrix},
  \end{equation*}
  and note that all columns of $M'$ sum up to exactly one, or in other words, its transpose is a stochastic matrix.
  Seeing $(M')^{t}$ as the transition matrix of a Markov chain on $\{1, 2, 3, 4\}$, state $4$ is always absorbing and (at $p_\star$) all other states $\{1, 2, 3\}$ are connected to $4$ through some path of length at most $2$.
  This means that $(M')^2$ has only positive entries on its last row and its columns sum up to one.
  Therefore the columns of $M^2$ sum up to values strictly smaller than $|\mathcal{E}_1|^2$.
  We conclude by observing that the Perron-Frobenius eigenvalue of $M$ is at most the maximum value of the square root of such sums.
  This means that $d_f/|\Edge_1| < 1$ as desired.

  We now prove that $d_f > \deg^{\mathrm{out}}(a_1)$. Start by noticing that
  \begin{equation}
    (1, 0, 1) \cdot M = \Big(
    \sum_{e' \in \Edge_{1}} \mathbb{P}^{1, e'}_{p_\star} (e'_- \leftrightarrow a_1),
    \quad
    \sum_{e' \in \Edge_{1}} \mathbb{P}^{1, e'}_{p_\star} (e'_+ \leftrightarrow a_1),
    \quad
    \sum_{e' \in \Edge_{1}} \mathbb{P}^{1, e'}_{p_\star} (e' \leftrightarrow a_1)
    \Big).
  \end{equation}
  By the Collatz-Wielandt characterization of the Perron-Frobenius eigenvalue, it follows that $d_f = \max_{x \neq 0; x_i \geq 0} \min_{i; x_i > 0} (xM)_i/x_i$, so that $d_f$ is at least
  \begin{equation*}
    \min \Big\{
    \sum_{e' \in \Edge_{1}} \mathbb{P}^{1, e'}_{p_\star} (e'_- \leftrightarrow a_1),
    \sum_{e' \in \Edge_{1}} \mathbb{P}^{1, e'}_{p_\star} (e' \leftrightarrow a_1)
    \Big\}
    = \sum_{e' \in \Edge_{1}} \mathbb{P}^{1, e'}_{p_\star} (e'_- \leftrightarrow a_1)
    > \deg^{\mathrm{out}}(a_1),
  \end{equation*}
  as desired.
  The proof that $d_f > \deg^{\mathrm{in}}(b_1)$ is completely analogous, while the final bound is obtained by replacing the vector $(1, 0, 1)$ with $(\alpha, \beta, \alpha + \beta)$, where $\alpha = \deg^{\mathrm{out}}(b_1)$ and $\beta = \deg^{\mathrm{in}}(a_1)$.

  Finally, to prove \cref{eq:one_arm_exponent}, observe that
  \begin{equation}
    \begin{split}
      \mathbb{P}^{k}_{p_{\star}} \big( E_k \leftrightarrow \{a_k, b_k\} \big)
      = \frac{1}{|\Edge_{1}|^k} \sum_{e \in \Edge_k}
        \mathbb{P}^{k}_{p_{\star}} \big( e \leftrightarrow \{a_k, b_k\} \big)
      = \frac{1}{|\Edge_{1}|^k}
        \mathbb{E}^{k}_{p_{\star}} \big[ |\mathcal{C}_{\{a_k, b_k\}}| \big]
    \end{split}
  \end{equation}
  The result then follows from \cref{eq:expected_volume_critical} and the definition of $\alpha_1$ in the statement of the theorem.
\end{proof}

The next result we state extends Theorem~\ref{th:volume-of-critical-cluster} by actually proving that the size of the critical cluster converges in distribution, when properly rescaled.
This is an extension of Section~3 in~\cite{hk2009}.

\begin{theorem}[Hambly-Kumagai]
  \label{t:volume_fluctuations}
  At the critical point $p_\star$,
  \begin{equation}
    \label{eq:volume-of-critical-cluster}
    \frac{|\mathcal{C}_{\{a_k, b_k\}}|}{d_f^k} \underset{k}{\Longrightarrow} \nu
  \end{equation}
  where the above denotes convergence in distribution towards a probability distribution $\nu$ supported on $(0, \infty)$.
\end{theorem}

\begin{remark}
  The above statement was obtained earlier in Hambly and Kumagai~\cite{hk2009} for the case of the diamond hierarchical lattice.
  Even though the two proofs are similar, we chose to write it down for arbitrary seed graphs $G_1$ (satisfying \eqref{e:hypothesis}) for the reader's convenience.
\end{remark}

In order to understand the scaling limit of $\mathcal{C}_{\{a_k, b_k\}}$, we will introduce an auxiliary Multi-Type Branching Process (MTBP).
But before doing so, we will need some further notation.
For $k \in \{0, 1, \dots\} \cup \{ \infty \}$, denote by $\mathcal{T}_k$ the set of words of length at most $k$, that is
\begin{equation}
  \label{eq:trees}
  \mathcal{T}_k = \mcup_{j = 0}^k \Edge_{1}^j, \qquad \text{and} \qquad
  \mathcal{T}_\infty = \mcup_{j \geq 0} \Edge_{1}^j.
\end{equation}
Recall that for $w \in \Edge_{1}^j$, we write $|w| = j$ for its length.
We regard $\mathcal{T}_k$ as a tree of depth~$k$, where $\varnothing$ stands for the root and if $w' = we$, we say that $w'$ is a child of $w$.
Recalling the notation from \cref{sec:notation}, to each node $w \in \mathcal{T}_k$ we associate the subgraph $G_k^w \subseteq G_k$.
Observe finally that the leaves of $\mathcal{T}_k$ correspond to the set of edges of $G_k$.

For each node $w \in \mathcal{T}_k$, we associate a label $L^w \in \{0, 1\}^3$ (thus there are eight possible labels in total) which will be a function of the percolation configuration inside $G_k$, so one can see $L^w$ as a random variable in $\Omega_k$.
In order to define $L^w$, we will introduce now the notion of \emph{left connectivity}: the graph $G_k^w$ is \emph{left connected} to the boundary of $G_k$, if $a^{w}_{k}$, the point corresponding to $a_{|w|}$ in the graph $G_{k}^{w}$ is connected to $\{a_k, b_k\}$ while avoiding the copy $G_{k}^{w}$, that is
\begin{equation}
  \label{eq:left_connected}
  \Big[ a^{w}_{k} \xleftrightarrow{G_k \setminus G_k^w}
  \{a_k, b_k\} \Big],
\end{equation}
see illustration in~\cref{fig:left_connected}.
Note that this event is independent of the state of the edges inside $G_k^w$.
Analogously, we say that $G_k^w$ is \emph{right connected} to the boundary of $G_k$ if $b^{w}_{k}$ is connected to the boundary $\{a_k, b_k\}$ of $G_k$ without using $G_k^w$ itself.

Given $w \in \mathcal{T}_k$ we define $L^w = (L^w_-, L^w_+, L^w_\circ)$, where
\begin{equation}
\begin{split}
  \label{eq:left_right_cross}
  L^w_- & := \mathbf{1}_{\big[ G_k^w
        \text{ is left connected to the boundary of $G_k$}, \big]},\\
  L^w_+ & := \mathbf{1}_{\big[ G_j^w
        \text{ is right connected to the boundary of $G_k$} \big]}, \text{ and}\\
  L^w_\circ & := \mathbf{1}_{\big[ G_k^w
        \text{ is crossed internally} \big]}.
\end{split}
\end{equation}
For example, we have $L^w = 000$ if none of the above events occurred and $L^{w} = 111$ if all of them did.

Observe that knowing the labels of all nodes at generation $j$ (that is, knowing $L^{w}$ for all words $w \in \mathcal{E}_{1}^{j}$), one is able to determine the labels of all nodes at previous generations $0, \dots, j - 1$.

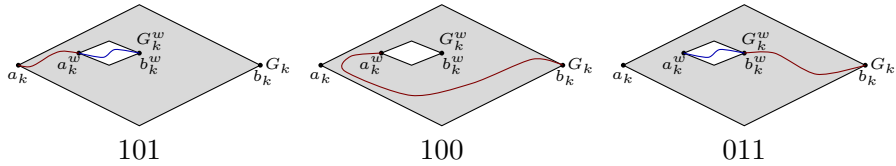
\begin{figure}[h]
  \centering
  \begin{tikzpicture}[scale = .8]
    \foreach \x in {0,...,2} {
      \begin{scope}[shift={(5*\x,0)}]
        \draw[fill=gray!30!white] (0, 0) -- (2, 1) -- (4, 0) -- (2, -1) -- (0, 0);
        \draw[fill=white] (1, .2) -- (1.5, .4) -- (2, .2) -- (1.5, 0) -- (1, .2);
        \node at (4.3, 0) {\tiny $G_k$};
        \node at (2.2, .4) {\tiny $G_k^w$};
        \fill (0, 0) circle (0.04);
        \node at (0, -.2) {\tiny $a_k$};
        \fill (4, 0) circle (0.04);
        \node at (4.05, -.2) {\tiny $b_k$};
        \fill (1, .2) circle (0.04);
        \node at (0.85, 0) {\tiny $a^{w}_{k}$};
        \fill (2, .2) circle (0.04);
        \node at (2.2, 0) {\tiny $b^{w}_{k}$};
      \end{scope}
    }
    \draw[color=red!50!black] (0, 0) .. controls ++(.2,-.05) .. (.5,.15)
    .. controls ++(.2,.1) .. (1,.2);
    \node at (2, -1.4) {$101$};
    \draw[color=red!50!black] plot [smooth, tension=10] coordinates
    { (9, 0) (8.5, .1) (7.5, -.3) (6.5, -.5) (5.4, -.1) (5.5,.15) (6,.2) };
    \draw[color=blue!70!black] plot [smooth, tension=10] coordinates
    { (1, .2) (1.4, .13) (1.6, .26) (2, .2) };
    \node at (7, -1.4) {$100$};
    \draw[color=red!50!black] plot [smooth, tension=10] coordinates
    { (12, .2) (12.5, .2) (13.2, -.15) (14, 0) };
    \draw[color=blue!70!black] plot [smooth, tension=10] coordinates
    { (11, .2) (11.4, .13) (11.6, .26) (12, .2) };
    \node at (12, -1.4) {$011$};
  \end{tikzpicture}
  \caption{An illustration of the left and right connected events, as well as the label of the boxes in three different situations.
    Note that only the gray region can be used in the connection from $a^{w}_{k}$ or $b^{w}_{k}$ to $\{a_k, b_k\}$ and only the white region can be used for the crossing.}
  \label{fig:left_connected}
\end{figure}

Now that we have introduced the labels of every $w \in \mathcal{T}_k$, our MTBP will be defined once we specify which nodes are alive or dead in $\mathcal{T}_k$:
\begin{display}
  \label{eq:mtbp_alive} A node $w \in \Edge^j_{1}$, for $j = 0, \dots, k$ is said to be alive if and only if the associated graph $G_k^w$ is either left or right connected.
  Otherwise it is said to be dead, which happens precisely if it is labeled $000$ or $001$.
\end{display}
In particular, at generation $j = 0$ the root $\varnothing$ is always alive and it has label either $111$ or $110$ depending on whether $G_k$ is internally crossed or not. This happens with probability $f^{k}(p)$ and $1-f^{k}(p)$, respectively.

\begin{remark}
  Observe that:
  \begin{enumerate}[\quad a)]
  \item if $L^w \in \{ 000, 001 \}$ (thus $w$ is dead), then it is not left nor right connected to $\{a_k, b_k\}$.
    In particular, all its subgraphs must have a label of type $000$ or $001$ as well.
    In other words, only alive individuals can have alive offspring.
  \item from \cref{eq:mtbp_alive} above, the alive individuals at the last generation ($j = k$) correspond precisely to $\mathcal{C}_{\{a_k, b_k\}}$.
    This is why understanding the asymptotic behavior of this MTBP will give us information about $|\mathcal{C}_{\{a_k, b_k\}}|$.
    \item Notice that since boundary edges are trivially connected to either $a_{k}$ or $b_{k}$, the MTBP never dies out.
  \end{enumerate}
\end{remark}

The process introduced above also considered by Hambly and Kumagai~\cite{hk2009}, where the authors study the fractal properties of critical percolation in the particular case of the diamond hierarchical lattice (see Figure~\ref{f:d_k}).

The next lemma is devoted to proving that the above process is indeed a MTBP.
Fix $0 \leq j < k$ and an assignment of labels to certain nodes of $\mathcal{T}_k$.
We say that a family of labels is \emph{consistent} if there is some configuration of open and closed edges of $G_k$ that induces it.

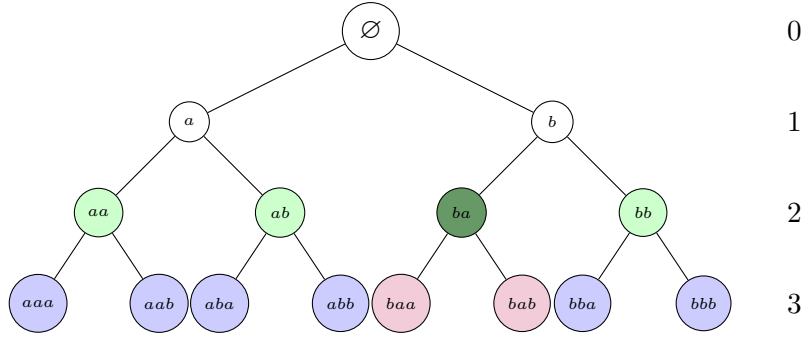
\begin{figure}
  \centering
  \begin{tikzpicture}[level/.style={sibling distance=60mm/#1},scale=.8]
    \node [circle,draw] (z){$\varnothing$}
    child {node [circle,draw] (a) {\tiny $a$}
      child {node [circle,draw,fill=verde] (aa) {\tiny $aa$}
        child {node [circle,draw,fill=azul] (aaa) {\tiny $aaa$}}
        child {node [circle,draw,fill=azul] (aab) {\tiny $aab$}}
      }
      child {node [circle,draw,fill=verde] (ab) {\tiny $ab$}
        child {node [circle,draw,fill=azul] (aba) {\tiny $aba$}}
        child {node [circle,draw,fill=azul] (abb) {\tiny $abb$}}
      }
    }
    child {node [circle,draw] (b) {\tiny $b$}
      child {node [circle,draw,fill=green!20!gray] (ba) {\tiny $ba$}
        child {node [circle,draw,fill=roxo] (baa) {\tiny $baa$}}
        child {node [circle,draw,fill=roxo] (bab) {\tiny $bab$}}
      }
      child {node [circle,draw,fill=verde] (bb) {\tiny $bb$}
        child {node [circle,draw,fill=azul] (bba) {\tiny $bba$}}
        child {node [circle,draw,fill=azul] (bbb) {\tiny $bbb$}
          child [grow=right] {node (q) {$3$} edge from parent[draw=none]
            child [grow=up] {node (r) {$2$} edge from parent[draw=none]
              child [grow=up] {node (s) {$1$} edge from parent[draw=none]
                child [grow=up] {node (t) {$0$} edge from parent[draw=none]
                }
              }
            }
          }
        }
      }
    };
  \end{tikzpicture}
  \caption{\cref{l:independent_cousins} gives us that, conditioned on the labels of green vertices, the red labels are independent of the blue ones \cref{eq:independent_families}.
  Moreover, the law of the red labels only depend on the green ones through the label of the dark green node \cref{eq:offspring_markov}.}
\end{figure}

Roughly speaking, the next lemma characterizes the distribution of labels at generation $j + 1$, given those at generation $j$.
This is done in two claims:
First, the offspring of a node at generation $j$ is independent of the others.
Second, their marginal distributions are given explicitly.

\begin{lemma}
  \label{l:independent_cousins}
  Fix $0 \leq j < k$, $p \in [0,1]$, and a consistent set of labels $\mathcal{L} = (l^w)_{w \in \Edge_{1}^j}$.
  Then
  \begin{enumerate}[\quad a)]
  \item ``The families produced by individuals at generation $j$ are independent'':
    \begin{equation}
      \label{eq:independent_families}
      \big( L^{we_1}, \dots, L^{we_{|\Edge_{1}|}} \big)_{w \in \Edge_{1}^{j}}
      \text{, are independent under }
      \mathbb{P}^{k}_{p} \big[ \cdot \big|
      L^w = l^w; \text{ for $w \in \Edge_{1}^{j}$} \big],
    \end{equation}
    see \cref{r:cousins}.
  \item ``The type of your father is the only relevant information about your previous generation'':\\
    Fixed $w_0 \in \Edge_{1}^{j}$ and some function $f: \{0, 1\}^{3 \times |\Edge_{1}|} \to \mathbb{R}$,
    \begin{equation}
      \label{eq:offspring_markov}
      \begin{split}
        \mathbb{E}_{p}^{k} \Big[ f \big( L^{w_0e_1}, & \dots, L^{w_0e_{|\Edge_{1}|}} \big)
                                                 \Big| L^w = l^w; \text{ for all $w \in \Edge_{1}^{j}$}
                                                 \Big]\\
        = & \mathbb{E}_{p}^{k} \Big[
            f \big( L^{w_0e_1}, \dots, L^{w_0e_{|\Edge_{1}|}} \big)
            \Big| L^{w_0} = l^{w_0} \Big],
      \end{split}
    \end{equation}
    see again \cref{r:cousins} for a discussion.
  \item ``The offspring distribution is homogeneous across generations at the critical point $p_{\star}$'':
    \begin{equation}
      \label{eq:offspring_rerooting}
      \mathbb{E}_{p_{\star}}^{k} \Big[
      f \big( L^{w_0e_1}, \dots, L^{w_0e_{|\Edge_{1}|}} \big)
      \Big| L^{w_0} = l^{w_0} \Big]
      = \mathbb{E}_{p_{\star}}^{k - j} \Big[
      f \big( L^{e_1}, \dots, L^{e_{|\Edge_{1}|}} \big)
      \Big| L^{\varnothing} = l^{w_{0}} \Big].
    \end{equation}
  \end{enumerate}
\end{lemma}

\begin{remark}
  \label{r:cousins}
  It is important to notice that:
  \begin{enumerate}[\quad a)]
  \item Given $w \in \Edge_{1}^{j}$, the vector appearing in \eqref{eq:independent_families} labels the children of $w$.
    It is not expected that $L^{we}$ be independent of $L^{we'}$, but rather the vector $(L^{we_1}, \dots, L^{we_{|\Edge_{1}|}})$ is independent of $(L^{w'e_1}, \dots, L^{w'e_{|\Edge_{1}|}})$, given the label of the previous scale 	$w' \neq w$.
  \item The identity in \cref{eq:offspring_markov} is reminiscent of the domain Markov property.
    Intuitively speaking, it says that knowing the label of $w_0$ is enough to learn the distribution of its offspring and there is no more information to be gained by learning other labels at the same generation as that of $w_0$.
   \item Finally, Equation~\eqref{eq:offspring_rerooting} states that the MTBP introduced above is homogeneous at the critical point $p_{\star}$, strengthening the relation $f(p_{\star}) = p_{\star}$ obtained in Theorem~\ref{t:crossing}.
  \end{enumerate}
\end{remark}

\begin{proof}[Proof of \cref{l:independent_cousins}]
Start by observing that, given the collection of labels in some scale $j<k$, the percolation configuration at scale $k$ can be recovered by independently sampling, in each graph $G^{w}_{k}$ a configuration at density $p$ conditioned on the crossing event $C_{k-j}$ or its complement, a choice that is determined by the last entry of the label $l^{w}$, for each subgraph. Furthermore, in order to determine the labels of a word $we \in \Edge_{1}^{j+1}$, it suffices to observe the label $l^{\omega}$ and the crossings of the graphs $G^{we'}_{k}$, for $e' \in \Edge_{1}$. These two observations already imply statements~\eqref{eq:independent_families} and~\eqref{eq:offspring_markov}.

We are left with verifying~\eqref{eq:offspring_rerooting}. This is a consequence of the fact that, at the critical point $p_{\star}$, crossings of graphs $G^{we'}_{k}$ happen with probability $p_{\star}$.
\end{proof}

From the lemma above we can now conclude the proof of Theorem~\ref{t:volume_fluctuations}.
\begin{proof}[Proof of~\cref{t:volume_fluctuations}]
  From Lemma~\ref{l:independent_cousins}, the MTBP~\eqref{eq:mtbp_alive} is homogeneous through the generations, since we are considering the process at $p_{\star}$.
  The proof of the theorem will follow from the Kesten-Stigum Theorem for supercritical multitype branching processes.
  
  Let $M$ denote the matrix of expected offspring sizes for types $l$ that are not $000$ and $001$, that is, $M(l, l')$ is the expected number of individuals with type $l'$ originating from a parent of type $l$ (with respect to $\mathbb{E}^{1}_{p}$, as the process is homogeneous).
  Let us first notice that the matrix $M$ is irreducible, as $M^2$ has only positive entries. In particular, we can apply Perron-Frobenius Theorem to the matrix $M$. Denote by $\rho$ the largest eigenvalue of $M$.
  Our next step is to prove that $\rho > 1$.
  
  Given a label $\ell \in \{0, 1\}^3 \setminus \{ 000, 001\}$, we write $Z_k^\ell$ for the number of nodes at generation $j=k$ with label $\ell$, and by $\vec{Z}_{k}$ the vector counting the number of individuals of each type in generation $k$.
  
  Notice that
  \begin{equation}
  \E \big( \vec{Z}_{k+1} \big) = M \E \big( \vec{Z}_{k} \big),
  \end{equation}
  for all $k \geq 0$. In particular,
  \begin{equation}
  \big\lVert \E \big( \vec{Z}_{k} \big) \big\rVert \leq \rho^{k} \big\lVert \E \big( \vec{Z}_{0} \big) \big\rVert.
  \end{equation}
  
  Since
  \begin{equation}\label{eq:expression_C_k}
    |\mathcal{C}_{\{a_k, b_k\}}|
    = Z_k^{100} + Z_k^{101} + Z_k^{010} + Z_k^{011} + Z_k^{110} + Z_k^{111},
  \end{equation}
  it follows from Theorem~\ref{th:volume-of-critical-cluster} that, for all $k \geq 1$,
  \begin{equation}
   \uc{c:volume_low} d_f^k
    \leq \mathbb{E}_{p_\star}^{k} \big[ |\mathcal{C}_{\{a_k, b_k\}}| \big] \leq \sqrt{6} \big\lVert \E \big( \vec{Z}_{k} \big) \big\rVert \leq \sqrt{6} \rho^{k} \big\lVert \E \big( \vec{Z}_{0} \big) \big\rVert.
  \end{equation}
  In particular, $\rho \geq d_{f}>1$.
  
  This allows us to apply Kesten-Stigum Theorem for multitype branching processes (see for example~\cite{ks66}) and conclude that
  \begin{equation}
  \frac{\vec{Z}_{k}}{\rho^{k}} \underset{k}{\Longrightarrow} \tilde{\mu}_{\infty},
  \end{equation}
for some non-trivial distribution $\tilde{\nu}$. Furthermore, $\tilde{\mu}_{\infty}(0) = 0$, as the process never dies out and the offspring distribution is bounded from above by $|\Edge_{1}|$.
  
  Once again using~\eqref{eq:expression_C_k}, we get
  \begin{equation}
  \frac{|\mathcal{C}_{\{a_k, b_k\}}|}{\rho^{k}} \underset{k}{\Longrightarrow} \mu_{\infty},
  \end{equation}
  for some distribution $\nu$ such that $\mu_{\infty}(0)=0$.
  
  It remains to prove that $\rho = d_f$. We already obtained $\rho \geq d_{f}$ in the above. For the reverse inequality, notice that~\cref{eq:expected_volume_critical} once again implies
  \begin{equation}
  0 < \int x \d \mu_{\infty}(x) \leq \liminf \mathbb{E}_{p_{\star}}^{k} \bigg[ \frac{|\mathcal{C}_{\{a_k, b_k\}}|}{\rho^{k}} \bigg] \leq \liminf \uc{c:volume_high} \frac{d_f^k}{\rho^{k}},
  \end{equation}
from which follows that $d_f \geq \rho$. This implies $d_f = \rho$ and concludes the proof.
\end{proof}

\section{Scaling Relations}
\label{s:exponent_relations}

Several critical exponents are expected to be related by what is now known as ``scaling relations.'' For planar percolation, many of these relations have been proved by Kesten under the assumption that the exponents exist, see~\cite{kesten1987scaling}.
For high dimensional percolation, the exponents take their mean-field values, see~\cite{hara1990mean}, implying the existence of such relations. In this section, we prove a relation between the percolation exponent, the one-arm explonent and the correlation length.

Recall the definition of the percolation function $\theta^{\bar{G}_{\infty}}$ immediately above~\eqref{eq:p_c} and set
\begin{equation}
\label{eq:theta}
\theta(p) = \E \big[ \theta^{\bar{G}_{\infty}}(p) \big],
\end{equation}
where $\E$ denotes the expectation with respect to the realization of the random limiting graph $\bar{G}_{\infty}$.
\nc{c:beta_and_relations}
\begin{theorem}
  \label{th:beta_and_relations}
  There exists $ \uc{c:beta_and_relations} > 0$ such that, for $p \in (p_{*}, 1)$,
  \begin{equation}
    \label{e:bound_theta_th}
    \uc{c:beta_and_relations}^{-1} (p - p_\star)^\beta
    \leq
    \theta(p)
    \leq
    \uc{c:beta_and_relations} (p - p_\star)^\beta,
  \end{equation}
  where
  \begin{equation}
    \label{eq:exponent_relation}
    \beta = \log_\zeta \frac{|\Edge_{1}|}{d_f} = \alpha_1 \nu.
  \end{equation}
\end{theorem}

\begin{proof}
  Recall the parameters $x_k^a, x_k^b$, and $x_k^{ab}$ defined in~\eqref{eq:three_expectations}, which will be denoted by $x = (x^{a}, x^{b}, x^{ab})$ in this section, and consider the following dynamical system, related to \cref{eq:x_ab_evolution}:
  \begin{equation}
    \label{eq:x_ab_evolution_2}
    \begin{split}
      & F: [0, 1] \times \mathbb{R}_+^3 \to [0, 1] \times \mathbb{R}_+^3
        \qquad  \text{given by} \\
      & F(p, x) = \big( f(p), \tfrac{1}{|\Edge_{1}|} M(p) x \big).
    \end{split}
  \end{equation}
  This system will be central to the understanding of the critical exponent $\beta$.
  We start by making a few observations about $F$.
  First, the octants $\{0\} \times \mathbb{R}_+^3$, $\{p_\star\} \times \mathbb{R}_+^3$ and $\{1\} \times \mathbb{R}_+^3$ are invariant under $F$.
  Also,
  \begin{equation}
    \label{eq:M_0_M_1}
    M(0) =
    \begin{pmatrix}
      \deg_1^{\mathrm{out}}(a_1) & \deg_1^{\mathrm{in}}(a_1) & \deg_1(a_1) \\
      \deg_1^{\mathrm{out}}(b_1) & \deg_1^{\mathrm{in}}(b_1) & \deg_1(b_1) \\
      0                          & 0                         & 0
    \end{pmatrix},
    \quad
    M(1) =
    \begin{pmatrix}
      0       & 0       & 0 \\
      0       & 0       & 0 \\
      |\Edge_{1}| & |\Edge_{1}| & |\Edge_{1}|
    \end{pmatrix},
  \end{equation}
  and the principal eigenvalue of $M(0)$ is strictly smaller than $\deg_1(a_1) + \deg_1(b_1) \leq |\Edge_{1}|$.
  Meanwhile, the principal eigenvalue of $M(p_\star)$ is $d_f$, which by~\cref{th:volume-of-critical-cluster} is also strictly smaller than $|\Edge_{1}|$.
  Therefore, the fixed points of $F$ are precisely:
  \begin{display}
    \label{eq:F_fixed} $\big( 0, (0, 0, 0) \big)$, $\big( p_\star, (0, 0, 0) \big)$ and\\
    all the points in the line $\{(1, (a, a, a)) : a \geq 0\}$.
  \end{display}
  We will focus our attention mainly on $\big( p_\star, (0, 0, 0) \big)$.

  To understand the exponent $\beta$, we analyze the infinite product
  \begin{equation}
    \label{eq:J_p}
    J(p) = \prod_{k = 0}^\infty \frac{1}{|\Edge_{1}|} M \big( f_{k} (p) \big),
  \end{equation}
  since, as we argue below, for $p > p_{\star}$, $\theta(p) = (0,0,1) J(p)^{\intercal}(1, 1, 1)^\intercal$.
  
  In order to prove the statement above, define $y(p) = (y^{a}(p), y^{b}(p), y^{ab}(p))$, similarly to~\eqref{eq:three_expectations}, as
  \begin{equation}
  \label{eq:three_expectations_2}
  \begin{split}
    y^a(p) & = \PP_{p}^{\bar{G}_{\infty}} \Big( E^{-} \overset{\bar{G}_{\infty} \setminus \{E\}}{\longleftrightarrow} \infty \Big), \\
    y^b(p) & = \PP_{p}^{\bar{G}_{\infty}} \Big( E^{+} \overset{\bar{G}_{\infty} \setminus \{E\}}{\longleftrightarrow} \infty \Big), \text{ and} \\
    y^{ab}(p) & = \PP_{p}^{\bar{G}_{\infty}} \Big( E \overset{\bar{G}_{\infty} \setminus \{E\}}{\longleftrightarrow} \infty \Big).
    \end{split}
\end{equation}
  The recursive nature of the graph $\bar{G}_{\infty}$ (see Proposition~\ref{p:benjamini_schramm_marks}) implies that
  \begin{equation}
  y(p)^{\intercal} = \frac{1}{|\Edge_{1}|} M(p) y(f(p))^{\intercal}.
  \end{equation}
  Furthermore, immediately from the definition one gets
  \begin{equation}
  \theta(p) = \E \big[ y^{ab}(p) \big] = (0,0,1) \frac{1}{|\Edge_{1}|} M(p) \E \big[ y(f(p))^{\intercal} \big].
  \end{equation}
  Repeated iterations of the equation above, together with the fact that, for all $p>p_{\star}$, $\lim_{k} \E \big( y(f_{k}(p)) \big) = (1,1,1)$, yields
  \begin{equation}
  \theta(p) = (0,0,1) J(p)(1,1,1)^{\intercal},
  \end{equation}
  proving our claim.
  
  Let us now go back to~\eqref{eq:J_p}. We start by fixing some value $p > p_\star$ and studying the sequence of matrices
  \begin{equation}
    M \big( f_{k} (p) \big) = M(1) + \Delta_k(p), \text{ where }
    \Delta_k(p) = M \big( f_{k} (p) \big) - M(1).
  \end{equation}
  Observe from~\cref{th:off-critical-decay} that, for any $p_{0} \in (p_{\star}, 1)$, there exist $c,c'>0$ such that $\lVert \Delta_k(p) \rVert \leq c e^{-c' \cut(G_{1})^k}$, for all $p \geq p_{0}$ where $\lVert \cdot \rVert$ denotes the operator norm.
  In particular, $\sum_{k} \lVert \Delta_k(p) \rVert < \infty$. By further increasing the value of $p_{0}$, we can assume
  \begin{equation}
    \label{eq:sup_Delta}
    \sup_{p \geq p_0} \sum_{k} \lVert \Delta_k(p) \rVert < 1.
  \end{equation}

 To make the product in~\eqref{eq:J_p} rigorous, we will use classical results about convergence of matrix products.
  First, observe that:
  \begin{enumerate}[\quad a)]
  \item since $\frac{1}{|\mathcal{E}_{1}|} M(1)$ is idempotent, the product defining $J(p)$ converges for every $p > p_\star$, see for instance Theorem~2.1 of \cite{ARTZROUNI198611};
  \item for $p \geq p_0$, the function $J(p)$ is continuous with respect to the operator norm, see Corollary~2.1 of \cite{ARTZROUNI198611};
  \item consequently, $J(p)$ converges to $\frac{1}{|\mathcal{E}_{1}|} M(1)$ as $p$ goes to one;
  \item moreover, there exists $p_1 > p_0$, such that, for every $p \geq p_1$, each of the entries of $J(p)$ are larger than the corresponding one in $\frac{1}{2|\mathcal{E}_{1}|} M(1)$ and smaller than $2$.
  \end{enumerate}

  We now observe that for every $p \in (p_\star, p_0)$, there exists a single $k(p)$ such that
  \begin{equation}
    \label{eq:m_of_p}
    f_{k}(p) \in [p_1, f(p_1)).
  \end{equation}
  In fact, the relation $h \circ f = \zeta \cdot h$ tells us that
  \begin{equation}
    \label{eq:k_of_p}
    k(p) = \inf\{k : \zeta^k \cdot h(p) \geq h(p_1)\}
    = \lceil \log_\zeta(h(p_1)) - \log_\zeta(h(p)) \rceil.
  \end{equation}
  We now split the evolution of $\big( p, (1, 1, 1) \big)$ into the piece before and after $k(p)$.
  For this we define $\big( a(p), (x^{a}(p), x^{b}(p), x^{ab}(p)) \big) := F^{(k(p))} \big(p, (1, 1, 1) \big)$.
  Then, see that $\theta(p)$ is given by
  \begin{equation}
    \label{eq:Jp_split}
    \begin{split}
      (0,0,1) J(p) (1, 1, 1)^\intercal
      & = (0,0,1) J \big( a(p) \big) \cdot \bigg( \prod_{k = 0}^{k(p)}
        \frac{1}{|\mathcal{E}_{1}|} M \big( f_{k} (p) \big)^{\intercal} \bigg) (0, 0, 1)^\intercal\\
      & = (0,0,1) J \big( a(p) \big) (x^{a}(p), x^{b}(p), x^{ab}(p))^\intercal.
    \end{split}
  \end{equation}
  Recalling that $a(p) \in [p_1, f(p_1))$, we can use the above definition of $p_1$ to conclude that
  \begin{equation}
    \label{eq:bound_theta}
    \frac{1}{2} \big( x^{a}(p) + x^{b}(p) + x^{ab}(p) \big) \leq \theta(p) \leq 6 \big( x^{a}(p) + x^{b}(p) + x^{ab}(p) \big).
  \end{equation}
  Therefore, in order to understand the behavior of $\theta(p)$ as $p$ approaches $p_\star$, it is enough to understand the asymptotics of $x^{a}(p), x^{b}(p)$, and $x^{ab}(p)$.

  From \cref{eq:Jp_split}, we see that $(x^{a}(p), x^{b}(p), x^{ab}(p))$ equals
  \begin{equation}
    \label{eq:compare_with_Lp}
    \begin{split}
      \bigg( \prod_{k = 0}^{k(p)}
      \frac{1}{|\Edge_{1}|} M \big( f_{k} (p) \big) \bigg) (1, 1, 1)^\intercal
      & = \bigg(  (1, 1, 1) \bigg( \prod_{k = 0}^{k(p)}
        \frac{1}{|\Edge_{1}|} M \big( f_{k} (p) \big) \bigg)^\intercal \bigg)^\intercal\\
      & = \bigg(  (1, 1, 1) \bigg( \prod_{k = 0}^{k(p)}
        \frac{1}{|\Edge_{1}|} M \big( f_{k}^{-1} \big( a(p) \big) \big)^{\intercal} \bigg) \bigg)^\intercal\\
      & = \Big( \frac{d_f}{|\mathcal{E}_{1}|} \Big)^{k(p)} \big( (1, 1, 1) \cdot L(p) \big)^\intercal,
    \end{split}
  \end{equation}
  where
  \begin{equation}
    \label{eq:Lp}
    L(p) = \bigg( \prod_{k = 0}^{k(p)}
    \frac{1}{d_f} M \big( f_{k}^{-1} \big( a(p) \big) \big)^{\intercal} \bigg).
  \end{equation}
  In order to conclude the proof, it suffices to prove that $L(p)$ has coordinates bounded from above and uniformly positive for $p \in [p_{\star}, p_{0})$. In fact, we will verify that
  \begin{equation}
    \label{eq:Lnp}
    L(n, p) = \bigg( \prod_{k = 0}^{n}
    \frac{1}{d_f} M \big( f_{k}^{-1} \big( p \big) \big)^{\intercal} \bigg)
  \end{equation}
has uniformly positive entries for all $n \geq 1$ and all $p \in [p_{\star}, 1-\bar{\delta})$, for any $\bar{\delta}>0$.

Our first step in proving the statement above is establishing that it holds for $p=p_{\star}$.
Observe that the matrix $M_\star := (1/d_f) M(p_\star)^{\intercal}$ is a Perron-Frobenius matrix with principal eigenvalue one.
  Moreover, its corresponding left eigenvector $v$ and right eigenvector $w$ have positive entries and $\lim_{k} L(k, p_{\star}) = \lim_k M_\star^k = w v^{\intercal}$ (where $v$ and $w$ have been normalized in a way that the sum of the entries of $v$ is one and $v^\intercal w = 1$). This yields that $\lim_{k} L(k, p_{\star})$ has positive entries. Since the entries of $M_{\star}$ are positive and bounded, we conclude that there exist $0< a < b$ such that $L(n, p_{\star})_{i,j} = \big(M_{\star}^{n}\big)_{i,j} \in [a,b]$, for all $n$ and all $i,j =1,2,3$.
  
  Fix now $\varepsilon>0$ such that, if $M$ is a matrix with $\lVert M-M_{\star}^{n} \rVert < \varepsilon$, then $M_{i,j} \in \big[\frac{a}{2}, 2b \big]$, for all $i,j=1,2,3$.
  
In order to extend the the above to the whole interval $p \in [p_{\star}, 1-\bar{\delta})$, we start by defining
\begin{equation}
g_{n}(A_{1}, \dots, A_{n}) = \prod_{k = 0}^{n}
    \frac{1}{d_f} \big( M(p_\star)^{\intercal} + A_{k} \big).
\end{equation}
 We will consider the alternative (but equivalent) norm over the space of matrices, given by $\lVert M \rVert_v := \sup_{x \neq 0} |Mx|_v/|x|_v$, where $|x|_v := \max_i |x_i|/v_i$.
  This norm satisfies $\lVert AB \rVert_v \leq \lVert A \rVert_v \lVert B \rVert_v$ (since $|ABx|_v \leq \lVert A \rVert_v |Bx|_v \leq \lVert A \rVert_v \lVert B \rVert_v |x|_v$).
  Moreover, $\lVert M_\star \rVert_v = 1$ (the suppremum being attained at the Perron-Frobenius eigenvector).
  
  The two properties above are required in order to apply Corollary~2.1 of \cite{ARTZROUNI198611} for the fuctions $g_{n}$ and implies that this collection of functions is equicontinuous. In particular, for $\varepsilon>0$ as above, there exists $\delta>0$ such that, for all $n \in \N$,
  \begin{equation}
  \text{if } \sum_{k=1}^{n} \lVert A_{k} \rVert_v < \delta, \quad \text{then } \lVert g_{n}(A_{1}, \dots A_{n}) - M_{\star}^{n} \rVert < \varepsilon.
  \end{equation}
  
Define now
  \begin{equation}
    \Gamma_k(p) = M \big( f_{k}^{-1}(p) \big)^{\intercal} - M(p_\star)^{\intercal},
  \end{equation}
  And notice that, by writing $f= h^{-1} \circ \zeta \circ h^{-1}$, using the facts that $h$ and $h^{-1}$ are Lipschitz in compact intervals and that $p \mapsto M(p)$ is continuous, we obtain that there exist positive constants $c$ and $c'$ such that $\lVert \Gamma_k(p) \rVert_{v} \leq c e^{-c' k}$, for all $k \geq 0$ and all $p \in (p_{\star}, 1-\delta)$. In particular, if $k_{0}$ is chosen large enough, then
  \begin{equation}
    \sum_{k=k_{0}}^{+\infty} \lVert \Gamma_{k}(p) \rVert_v < \delta, \quad \text{for all } p \in [p_{\star}, 1-\bar{\delta}).
  \end{equation}
  
The choice of $\varepsilon$ and equicontinuity of the functions $g_{n}$ now imply that
  \begin{equation}
    \bigg(\prod_{k=k_{0}}^{n} \frac{1}{d_{f}} M \big( f_{k}^{-1} \big( p \big) \big)^{\intercal} \bigg)_{i,j} \in \Big[ \frac{a}{2}, 2b \Big],
  \end{equation}
for any $i,j=1,2,3$, $n \geq k_{0}$, and $p \in [p_{\star}, 1-\bar{\delta})$.

By continuity, the entries of $\bigg(\prod_{k=1}^{k_{0}-1} \frac{1}{d_{f}} M \big( f_{k}^{-1} \big( p \big) \big)^{\intercal} \bigg)_{i,j}$ are bounded from above and below uniformly for $p \in [p_{\star}, 1-\bar{\delta})$. These two facts imply that the same holds for the matrices $L(n,p)$ and concludes the proof of our claim.

From the above, \cref{eq:bound_theta}, and \cref{eq:compare_with_Lp}, we conclude that
  \begin{equation}
    \label{eq:2}
    c \Big( \frac{d_f}{|\mathcal{E}_{1}|} \Big)^{k(p)} \leq \theta(p) \leq c' \Big( \frac{d_f}{|\mathcal{E}_{1}|} \Big)^{k(p)},
  \end{equation}
which implies the theorem using \cref{eq:k_of_p}.
\end{proof}

\section{Near-critical scaling}
\label{s:near-critical}

For two-dimensional site percolation on the triangular lattice, Garban, Pete, and Schramm~\cite{garban2018scaling} obtained the near-critical scaling limit of the model.
Here we obtain results that are analogous to those for percolation on hierarchical lattices.

Recall $\zeta>1$ denotes the derivative of the crossing function $f$ at the fixed point $p_{\star}$ (see Equation~\eqref{eq:lower-bound-zeta}).
\begin{theorem}
  \label{th:near-critical-scaling}
  For any $\lambda \in \mathbb{R}$,
  \begin{equation}
    \label{eq:near-critical-scaling}
    f_{k} \Big( p_\star + \frac{\lambda}{\zeta^k} \Big) \to h^{-1}(\lambda),
  \end{equation}
  where the diffeomorphism $h: (0, 1) \to \mathbb{R}$ is characterized as the unique linear conjugation of $f$ (that is, $h \circ f = \zeta \cdot h$) such that $h'(p_\star) = 1$.

  \nc{c:tail_h_1}
  \nc{c:tail_h_2}
  Moreover, there exist constants $\uc{c:tail_h_1}, \uc{c:tail_h_2} > 0$ such that
  \begin{equation}
    \label{eq:tail_h_1}
    \exp \Big( -\uc{c:tail_h_1} |\lambda|^\nu \Big)
    \leq h^{-1}(\lambda)
    \leq \exp \Big( -\uc{c:tail_h_1}^{-1} |\lambda|^\nu \Big),
    \text{ as $\lambda \to -\infty$}
  \end{equation}
  and
  \begin{equation}
    \label{eq:tail_h_2}
    1 - \exp \Big( -\uc{c:tail_h_2} \lambda^\mu \Big)
    \leq h^{-1}(\lambda)
    \leq 1 - \exp \Big( -\uc{c:tail_h_2}^{-1} \lambda^\mu \Big),
    \text{ as $\lambda \to \infty$},
  \end{equation}
  where the exponents $\mu$ and $\nu$ are the ones appearing in \cref{th:off-critical-decay}.
\end{theorem}

\begin{remark}
  Define the flooding time $T_{k}$ as the random quantity $T_{k} = \inf \{p \in [0,1]: a_{k} \overset{\omega_{p}}{\longleftrightarrow} b_{k} \}$, where $\omega_{p}$ is the percolation configuration in $G_{k}$ obtained from the standard monotone coupling~\eqref{eq:percolation_k}.
  Theorem~\ref{th:near-critical-scaling} shows that $\zeta^k (T_k - p_\star)$ converges weakly to a non-degenerate random variable with cumulative distribution function given by $h^{-1}$.
  See Ahlberg and Steif~\cite{ahl_ste} for more details on the existence of such limits for several classes of Boolean functions.
\end{remark}

\begin{proof}
  We know by Hartmann-Grobman Theorem that there exists a $C^\infty$ diffeomorphism $h: (0 , 1) \to (- \infty, \infty)$ such that $h \circ f = \zeta \cdot h$ (recall $\zeta = f'(p_\star)$).
  Moreover, there is a unique choice for $h$ that satisfies $h'(p_\star) = 1$ (see Appendix~\ref{app:auxiliary}).

  From the relation $h \circ f = \zeta \cdot h$ follows that $h(p_{\star}) = 0$. In particular, Taylor expansion of $h$ around $p_{\star}$ yields
  \begin{equation}
  h(p_{\star}+\varepsilon) = \varepsilon + o(\varepsilon).
  \end{equation}
  Therefore,
  \begin{equation}
    \begin{split}
     f_{k} \Big( p_{\star} + \frac{\lambda}{\zeta^k} \Big) & = f^{(k)} \Big( p_{\star} + \frac{\lambda}{\zeta^k} \Big) \\
      & = h^{-1} \Big( \zeta^k h \Big( p_{\star} + \frac{\lambda}{\zeta^k} \Big) \Big) \\
      &  = h^{-1} \Big( \zeta^k \Big( \frac{\lambda}{\zeta^k} + o\Big( \frac{\lambda}{\zeta^k} \Big) \Big) \Big)\\
      & = h^{-1} \Big( \lambda + \zeta^k o \Big( \frac{\lambda}{\zeta^k} \Big) \Big)
        \xrightarrow{k \to \infty} h^{-1}(\lambda),
    \end{split}
  \end{equation}
  proving \cref{eq:near-critical-scaling}.

  We now study the asymptotic behavior of $h^{-1}(\lambda)$ as $\lambda \to -\infty$.
  For this, fix $p < p_\star$ such that $h(f(p)) \leq -\zeta$. Due to the limit~\eqref{eq:off-critical-decay-sub}, there exists $c(p) > 0$ such that
  \begin{equation}
  -c(p) \d_{1}(a_{1}, b_{1})^{k} \leq \log f_{k}(p), \quad \text{for all } k \in \N.
  \end{equation}

Recall now from~\eqref{eq:critical_exponent_correlation} that $\nu = \log_{\zeta} \d_{1}(a_{1}, b_{1})$ to obtain the bound, for $\lambda<0$,
  \begin{equation}
    \begin{split}
     h \Big( e^{-c(p) |\lambda|^\nu} \Big)
      & \leq h \Big( e^{-c(p)
        \; \d_{G_1}(a_1, b_1)^{\lfloor \log_\zeta|\lambda| \rfloor}} \Big)
        \leq
        h \Big( e^{
        \log f_{\lfloor \log_\zeta|\lambda| \rfloor}(p) } \Big)\\
      & = h \Big( f_{\lfloor \log_\zeta|\lambda| \rfloor} (p) \Big)
        = \zeta^{\lfloor \log_\zeta|\lambda| \rfloor} h \circ f (p)
        \leq -\zeta^{ \log_\zeta|\lambda|} \leq -|\lambda| = \lambda,
    \end{split}
  \end{equation}
  proving the first inequality of \cref{eq:tail_h_1}.
  All the other three bounds are analogous and we omit their proof.
\end{proof}

\section{Locality of critical threshold}
\label{s:locality}

We now investigate within the context of hierarchical graphs the common belief among physicists concerning locality and universality in percolation.
More precisely, it is believed that the critical threshold $p_c$ depends solely on the local geometry of a graph, while the critical exponents (such as $\zeta$, $\nu$, and $\beta$) depend only on the global structure of the graph near infinity.

To make this precise, let us consider two graphs $G^l$ and $G^g$ with respective functions $f_l$ and $f_g$ (here $l$ stands for \emph{local} and $g$ stands for \emph{global}).
Assuming that they both satisfy the hypotheses \cref{e:hypothesis}, we conclude form \cref{t:crossing} that the function $f_l$ has exactly three fixed points in $[0, 1]$, namely $0$, $1$ and $p_l \in (0, 1)$ (the same can be said about $f_g$ and $p_g$).

Let us also introduce the operator $\Gamma_l$ that given an oriented graph $G$ produces the graph $\Gamma_l(G)$ obtained by replacing all edges $e = (e_-, e_+)$ of $G$ with a copy of $G_l$, where the endpoints $e_-$ and $e_+$ are respectively attached to $a$ and $b$ in $G$.
Analogously we define $\Gamma_g$ that replaces edges by $G_g$.

Fix now two sequences of positive integers $l_k$ and $g_k$ growing to infinity.
Now consider for each $k$ the graph $G_k$ given by:
\begin{equation}
G_k = \Gamma_l^{(l_k)} \circ \Gamma_g^{(g_k)}(G),
\end{equation}
where G is the graph containing a single edge $\{a_0, b_0\}$.

In other words, in the global picture, we have $g_k$ iterations of $G_g$ and each of its edges is replaced by the local picture, composed of $l_k$ iterations of $G_l$.

We also introduce the constants $\zeta_l = f_l'(p_l)$ and $\zeta_g = f_g'(p_g)$, which are both strictly larger than one by \cref{t:crossing}.
Consider also the conjugations $h_l$ and $h_g$ as in Theorem~\ref{th:near-critical-scaling}.

Analogously to~\eqref{eq:C_k_f_k}, we introduce the crossing probability function $g_k(p) := \P_p [a_k \overset{G_k}\leftrightarrow b_k]$, that can be written as
\begin{equation}
  \label{eq:crossing_locality}
  g_k (p) = f_g^{(g_k)} \circ f_l^{(l_{k})} (p).
\end{equation}

We can now state the main result of this section.
\begin{theorem}
  \label{t:locality}
  There exist sequences $p_k \in (0, 1)$ and $\zeta_{k} \in \R$ such that, for any $\lambda \in \mathbb{R}$,
  \begin{equation}
    \label{eq:locality_fluctuations}
    g_k \Big( p_k + \frac{\lambda}{\zeta_k} \Big) \to h_g^{-1}(\lambda),
  \end{equation}
  (contrast this with \cref{eq:near-critical-scaling}).
  Moreover,
  \begin{enumerate}[\quad a)]
  \item ``The critical threshold is local'':\\
  The sequence $p_{k}$ converges to $p_{l} = p_c(G_1)$ as $k$ grows.
  
  \item ``Critical exponents are global'':\\
  If $l_k \ll g_k$, then
    \begin{equation}
      \label{eq:zeta_locality}
      \lim_k \frac{1}{l_k + g_k} \log \zeta_{k} = \log \zeta_g.
    \end{equation}
  \end{enumerate}
\end{theorem}

\begin{proof}
  We start by defining $p_{k} = f_{l}^{-(l_{k})}(p_{g})$, which converges to $p_{l}$ by the fact that $\zeta_{l} >1$.
  Let $A_{k} = (f_{l}^{l_{k}})'(p_{k})$ and define
  \begin{equation}\label{eq:zeta_k}
  \zeta_{k} = A_{k} \zeta_{g}^{n_{k}}.
  \end{equation}
  
  In order to prove~\eqref{eq:zeta_locality}, first observe that, since $f'$ is bounded, one gets $A_{k} \leq C^{l_{k}}$, for some $C>0$. From this,
  \begin{equation}
  \lim_k \frac{1}{l_k + g_k} \log \zeta_{k} = \lim_{k} \frac{l_{k}}{l_k + g_k} \log C + \frac{g_{k}}{l_k + g_k} \log \zeta_{g} = \log \zeta_{g}.
  \end{equation}
  
  We are now left with proving~\eqref{eq:crossing_locality}. Start by noticing that first-order Taylor expansion of $f^{(l_{k})}$ yields
  \begin{equation}\label{eq:locality_1}
  \begin{split}
  f_{l}^{(l_{k})} \Big( p_{k} + \frac{\lambda}{\zeta_{k}} \Big) & = f_{l}^{(l_{k})} (p_{k}) +  (f_{l}^{(l_{k})})' (p_{k}) \frac{\lambda}{\zeta_{k}} + R_{k}(\lambda) \\
  & = p_{g} + \frac{\lambda}{\zeta_{g}^{n_{k}}} + R_{k}(\lambda),
  \end{split}
  \end{equation}
where the error term satisfies
  \begin{equation}
  |R_{k}(\lambda)| \leq \frac{c^{l_{k}}}{\zeta_{k}^{2}}\lambda^{2},
  \end{equation}
for some positive constant $c$.
In particular,
\begin{equation}
\zeta_{g}^{n_{k}}|R_{k}(\lambda)| = o(1),
\end{equation}
for any fixed $\lambda \in \R$.

We now consider the function $f_{g}$ and write $h_{g}$ to be the Hartman-Grobman conjugation at $p_{g}$ with $h'(p_{g})=1$. Taylor expansion of $h_{g}$ around $p_{g}$ yields
\begin{equation}
h_{g} \Big( p_{g} + \frac{\lambda}{\zeta_{g}^{n_{k}}} + R_{k}(\lambda) \Big) = h_{g} \Big( p_{g} + \frac{\lambda + o(1)}{\zeta_{g}^{n_{k}}} \Big) = \frac{\lambda + o(1)}{\zeta_{g}^{n_{k}}}.
\end{equation}
This now implies
\begin{equation}
f_{g}^{(n_{k})} \Big( p_{g} + \frac{\lambda}{\zeta_{g}^{n_{k}}} + R_{k}(\lambda) \Big) = h_{g}^{-1} \big( \lambda + o(1) \big).
\end{equation}
The proof is complete by combining the equation above with~\eqref{eq:locality_1}.
\end{proof}

\section{Fixed points of Booelan functions}
\label{s:fixed_points}

Throughout this section, we fix a monotone Boolean function $g:\{0,1\}^{n} \to \{0,1\}$. We assume that $g$ is nontrivial, in the sense that $g(0)=0$ and $g(1)=1$. In this context, Russo's formula states
\begin{equation}
\frac{\d}{\d p} \mathbb{E}_{p}[g] = \mathbb{E}_{p}\big[ \#\{\text{pivotal bits for } g\} \big] = \sum \mathbb{P}_{p} ( j \text{ is pivotal for } g).
\end{equation}
In the particular case where $p=0$, the distribution $\mathbb{P}_{p}$ is deterministic and the derivative above counts the number of bits $j$ such that $g(e_{j}) = 1$, where $e_{j}$ denotes the configuration where all by the $j$th bit is equal to one. In particular, this derivative is always a non-negative integer. The same reasoning applied to the derivative at $p=1$, where it counts the number of bits $j \in [n]$ such that $g(1-e_{j})=0$.

Furthermore, if $\frac{\d}{\d p} \mathbb{E}_{p}[g] \Big|_{p=0} = \frac{\d}{\d p} \mathbb{E}_{p}[g] \Big|_{p=1} =1$, then $g$ is a dictator function. Indeed, since $g$ is monotone and its derivative at zero is one, there exists one bit $i \in [n]$ such that, if $\omega_i=1$, then $g(\omega)=1$. At the same time, an analogous argument using the derivative at one implies that there is another bit $j \in [n]$ such that, if $\omega_j=0$, then $g(\omega)=0$. In order to prove that $g$ is a dictator function, one just needs to observe that $i=j$. If this were not the case, by considering the configuration $\omega = e_i$, we would have $g(\omega)=1$, since $\omega_i=1$ and $g(\omega)=0$, as $\omega_j=0$, yielding a contradiction.

\begin{theorem}\label{t:fixed_poits}
Let $g:\{0,1\}^{n} \to \{0,1\}$ be a monotone non-trivial Boolean function that is not a dictator. Then the map $p \mapsto \mathbb{E}_{p}[g]$ has a fixed point in $(0,1)$ if, and only if, $\frac{\d}{\d p} \mathbb{E}_{p}[g]\Big|_{p=0} = \frac{\d}{\d p} \mathbb{E}_{p}[g]\Big|_{p=1} = 0$. Furthermore, this fixed point $p_{\circ}$ is unique and $\zeta = \frac{\d}{\d p} \mathbb{E}_{p}[g]\Big|_{p=p_{\circ}} >1$. 
\end{theorem}

\begin{remark}
The condition $\frac{\d}{\d p} \mathbb{E}_{p}[g] \Big|_{p=0} = \frac{\d}{\d p} \mathbb{E}_{p}[g] \Big|_{p=1} = 0$ might seem a bit hard to verify at first. One easy way to ensure it is to check that, if $e_{j} \in \{0,1\}^{n}$ denotes the configuration with all but the $j$th entry equal to zero, then $g(e_{j}) = 0$ and $g(1-e_{j}) = 1$, for all $j \in [n]$. This in fact implies that both derivatives are zero, by Russo's formula.
\end{remark}

The proof of the theorem above follows the same lines as the proof of Theorem~\ref{t:crossing}.
\begin{proof}
Let us first prove that, if a fixed point in the interval $(0,1)$ exists, the derivative of $\mathbb{E}_{p}[g]$ at this fixed point is strictly larger than one. This immediately implies uniqueness.
  
Consider the algorithm $\mathcal{A}_{1}$ that chooses one bit uniformly at random and reveals everything else. If the outcome of $g$ can then be determined, the random bit is not observed, otherwise, query this bit as well. For every bit $j \in [k]$, the revealment of the algorithm $\mathcal{A}_1$ is bounded by
\begin{equation}
\delta_{j}(p) \leq 1-\frac{1}{n}\Big(1- \max_{i \in [n]} \mathbb{P}_{p} ( i \text{ is pivotal for } g) \Big) <1.
\end{equation}
The last bound above follows from the fact that $g$ is not a dictator, which implies that the probability of each bit being pivotal is always smaller than one.

Denote now by $g_{k}:\{0,1\}^{kn} \to \{0,1\}$ the sequence of hierarchical Boolean functions obtained by successive applications of the function $g$ (notice that $g_{1} = g$). Let $\mathcal{A}_{k}$ be the recursive algorithm that determines $g_{k}$ by successively applying the algorithm $\mathcal{A}_{k-1}$ as needed. Notice that, if $p_{\circ} \in (0,1)$ is a fixed point of $p \mapsto \mathbb{E}_{p}[g]$, the revealment of $\mathcal{A}_{k}$ satisfies
\begin{equation}
\max_{i \in [kn]} \delta^{k}_{i}(p_{\circ}) \leq \Big(\max_{j \in [n]} \delta_{j}(p_{\circ}) \Big)^{k}.
\end{equation}

In particular, OSSS inequality~\eqref{eq:osss} applied to $g_{k}$ at $p_{\circ}$ yields
\begin{equation}
p_{\circ}(1-p_{\circ}) = \Var_{p_{\circ}}(g_{k}) \leq \sum_{i \in [kn]} \delta_i(p_{\circ}) \mathbb{P}_{p_{\circ}} ( i \text{ is pivotal for } g_{k}).
\end{equation}
By Russo's formula, we now obtain
  \begin{equation}
    \label{eq:largederivative2}
    \begin{split}
      \zeta^{k} = \frac{\d}{\d p} \mathbb{E}_{p}(g_{k}) \Bigg|_{p=p_{\circ}} & = \sum_{i \in [kn]} \mathbb{P}_{p_\circ} (i \text{ is pivotal to $g_k$}) \\
      & \geq p_{\circ}(1-p_{\circ}) \frac{1}{\max_{i \in [kn]} \delta^{k}_{i}(p_{\circ})} \\
      & \geq p_\circ (1 - p_\circ) \Big(\max_{j \in [n]} \delta_{j}(p_{\circ}) \Big)^{-k}.
    \end{split}
  \end{equation}
  Taking the limit as $k$ increases yields
  \begin{equation}
  \zeta \geq \Big(\max_{j \in [n]} \delta_{j}(p_{\circ}) \Big)^{-1}>1.
  \end{equation}

Let us now verify the condition for existence of the fixed point. Assume $\frac{\d}{\d p} \mathbb{E}_{p}[g] \Big|_{p=0} = \frac{\d}{\d p} \mathbb{E}_{p}[g] \Big|_{p=1} = 0$. Consider the map $h(p) = \mathbb{E}_{p}[g]-p$. Notice that $h(0)= h(1)=0$ and $h'(0) = h'(1) = -1$. This implies that the function $h$ is decreasing around $0$ and $1$. In particular, for $p$ small enough, $h(p)<0$ and $h(1-p)>0$. By continuity, there exists at least one value of $p_{\star} \in (0,1)$ such that $h(p_{\star}) = 0$, yielding the existence of the fixed point.

Suppose now that $\frac{\d}{\d p} \mathbb{E}_{p}[g] \Big|_{p=0} \geq 1$ and that $g$ is not a dictator function. Let us prove that there exists no fixed point in $(0,1)$. In this case, there exists at least one bit $i \in [n]$ such that, if $\omega_i=1$, then $\mathbb{E}_{p}[g] \geq \mathbb{P}_p [\omega_i=1] = p$. In particular, if $p_\circ \in (0,1)$ is a fixed point of $\mathbb{E}_{p}[g]$, necessarily $p \mapsto \mathbb{E}_{p}[g]$ is tangent to the identity function at $p_\circ$, implying $\frac{\d}{\d p} \mathbb{E}_{p}[g] \Big|_{p=p_\circ} = 1$, contradicting the first part of the proof. An analogous argument covers the case when $\frac{\d}{\d p} \mathbb{E}_{p}[g] \Big|_{p=1} \geq 1$ and concludes the proof.
\end{proof}

\begin{remark}
Theorem~\ref{t:sharp_ns} can also be written in terms of general Boolean functions. Assume $g:\{0,1\}^{n} \to \{0,1\}$ satisfies the hipotheses of Theorem~\ref{t:fixed_poits}. In this case, Theorem~\ref{t:sharp_ns} applies for the sequence $g_{k}: \{0,1\}^{nk} \to \{0,1\}$ of iterations of $g$, yielding sharp noise sensitivity in this case as well. In particular, we obtain an alternative proof of the noise sensitivity of the iterated majority function, and a precise description of the correlation limits depending on the decay rate of the noise $\varepsilon_{n}$ (see, for example~\cite{odonnell}).
\end{remark}

\section{Open problems}
\label{s:open}

Below we list a collection of problems and possible directions for further research on percolation and other models on hierarchical lattices.

\begin{enumerate}
\item {\bf Incipient infinite cluster.} Prove that the infinite incipient cluster in the limiting graph $G_{\infty}$ exists and can be realized in two possible (equivalent) ways.
\begin{itemize}
\item[1.a.] The infinite incipient cluster can be seen as the limiting distribution of supercritical percolation in $G_{\infty}$, conditioned on the event $[E_\infty \leftrightarrow \infty]$, as $p \downarrow p_c$.
\item[1.b.] Consider critical percolation on the random marked graph $(G_k, E_k)$, conditioned on the event $[E_k \leftrightarrow \{a_k, b_k\}]$. Prove that the local limit of the cluster of $E_{k}$ around its root converges to the incipient infinite cluster.
\end{itemize}

\item {\bf Exceptional times.} Consider dynamical percolation on the limiting graph $G_{\infty}$. Prove the existence of exceptional times and compute the dimension of the exceptional-time set.

\item {\bf One-arm scaling.} Simulations from~\cite{levitan} (see Figure 3 and Equation (11)) suggest that, for the diamond hierarchical lattice (see Figure~\ref{f:d_k}), we have a limit of the scaled one-arm probability.
  More precisely, for any $\lambda > 0$,
\begin{equation}
  \d_{1}(a_{1}, b_{1})^{\alpha_1 k} \;\; \mathbb{P}^{k}_{p_\star + \lambda/{\zeta^k}}
  \big( E_{k} \longleftrightarrow \{a_k, b_k\} \big)
  \overset{k}{\rightarrow} \eta(\lambda) \in (0, \infty).
\end{equation}
It would be interesting to prove this scaling limit both in the case of the diamond lattice and for other seed graphs.

\item {\bf Random seeds.} Consider two possible seed graphs $G$ and $H$. At each step of the construction of $G_{k}$, choose independently either $G$, with probability $q$, or $H$, with probability $1-q$, and perform the edge-replacing step with the chosen graph. Thus, $G_{k}$ is a random graph. Can the analysis performed here be carried out in this context? In particular, extrapolating from Section~\ref{s:locality}, we conjecture that, while the critical point will be random, critical exponents should depend on macroscopic scales, where averaging effects take place, and will thus be constant.

\item {\bf Other statistical mechanics models.} Consider different models atop of hierarchical lattices, such as dependent percolation, the Ising model, and FK percolation.
  Several works in this direction have been mentioned in \cref{ss:history}.
  However, just as it was done in the case of Bernoulli percolation in this article, it would be an interesting investigation to try to establish more general and refined results for such models.

\item {\bf Dependent percolation.} Recently, there has been a growing interest in percolation models featuring strong dependence, such as random interlacements \cite{Szn09}, level sets of the Gaussian Free Field \cite{rodriguez2013phase} and Boolean percolation \cite{meester1996continuum}.
It would be interesting to translate some of these models to the hierarchical context in a way that maintains their tractability and then try to understand the effect these dependencies have on the critical and near-critical behavior of the model.

\item {\bf Deviations for noise-sensitivity.} In a conversation with Gábor Pete, he suggested the following question.
  Suppose that, for each edge in $G_k$, one performs a continuous-time dynamics that re-samples states after exponential times of parameter one.
  Denote by $(\omega_t)_{t \geq 0}$ the trajectory of this process.
  For a fixed parameter $T$, we have established in Theorem~\ref{t:sharp_ns} the existence of limit $\mathbb{P}_{p_\star}(\omega_{0} \in C_{k}, \omega_{T/\zeta^{k}} \in C_k)$ as $k$ grows.
  It would be interesting to obtain the limit of $\mathbb{P}_{p_{\star}}( \omega_{t/\zeta^{k}} \in C_k, \text{ for every $0 \leq t \leq T$})$ as $k \to \infty$.
  Also, what is the asymptotic behavior of this limit as $T \to \infty$?
  See Section~4 of \cite{hammond_mossel_pete} for results and conjectures pertaining to the triangular lattice version.
\end{enumerate}

\appendix

\section{Auxiliary results}\label{app:auxiliary}
\label{s:auxiliary}

In this section we collect some results that are used throughout the manuscript. These are classical results in the analysis of hyperbolic fixed points for dynamic systems and can be found in any text book. See, for instance,~\cite{viana_espinar}.

\begin{theorem}[Hartman-Grobman]\label{t:hg}
Let $U, \tilde{U} \subset \R^{n}$ be open sets and consider a diffeomorphism $f:U \to \tilde{U}$. Assume there exists $u_{*} \in U$ such that $f(u_{*}) = u_{*}$ and that the Jacobian matrix $Df(u^{*})$ has no eigenvalue with absolute value one. Then there exist an open neighborhood $V \subset U$ of $u_{*}$, an open neighborhood $\tilde{V} \subset \R^{n}$ of the origin, and a homeomorphism $h: V \to W$ such that $h(u^{*}) = 0$ and
\begin{equation}\label{eq:conjugation}
f(x) = h^{-1}\Big( Df(u_{*}) h(x) \Big), \qquad \text{for all } x \in V.
\end{equation}
\end{theorem}

It is not the case that the homeomorphism $h$ needs to be differentiable. The following result gives us sufficient conditions to ensure this in the one-dimensional case.
\begin{proposition}\label{p:diffeo_conjugation}
Let $U, \tilde{U} \subset \R$ be open sets and consider a $C^{k+1}$ diffeomorphism $f:U \to \tilde{U}$, for some $k \geq 1$. Assume there exists $u_{\star} \in U$ such that $f(u_{\star}) = u_{\star}$ and such that $\zeta = f'(u_{\star})$ satisfies $|\zeta| \neq 1$. There exists a $C^{k}$ diffeomorphism $h: V \to \tilde{V}$ between open neighborhoods of $u_{\star}$ and the origin such that $h(u_{\star}) = 0$ and such that
\begin{equation}
f(x) = h^{-1} \Big( \zeta h(x) \Big), \quad \text{for all } x \in V.
\end{equation}
Furtheremore, the diffeomorphism $h$ is unique up to multiplication by constants.
\end{proposition}

In particular, as a consequence of the result above, as long as $f$ is at least twice-differenciable, the conjugation map $h$ is also a diffeomorphism and it can be chosen in a unique way with the added restriction that $h'(u_{\star}) = 1$.

Finally, let us mention a particular result that is used in the text. Assume $f:(0,1) \to (0,1)$ is a $C^2$ diffeormophim with a unique fixed point $u_{\star} \in (0,1)$ such that $\zeta = f'(u_{\star})>1$. In this case, we claim that the conjugation given by the Hartman-Grobman Theorem above can be extended in a way such that $h:(0,1) \to \R$.

Indeed, assume that one initially defines a conjugation $\tilde{h}: (a,b) \to (a',b')$, with $u_{\star} \in (a,b)$ and $a' < 0 < b'$. Since $\zeta >1$, $u_{\star}$ is a repulsive fixed point, meaning that $f^{-1}(a,b) \subset (a,b)$. Furthermore, from~\eqref{eq:conjugation}, one gets, for $y \in (a,b)$ and $m \in \N$ such that $f^{m}(y) \in (a,b)$,
\begin{equation}\label{eq:iterated_conjugation}
h \circ f^{m} (y) = \zeta^{m} \tilde{h} (y).
\end{equation}

We now use the fact that the fixed point $u_{\star}$ is unique. This implies that $\cap_{n \geq 0} f^{n}(a,b) = (0,1)$. Given $x \in (0,1)$, there exists $n \in \N$ such that $f^{-n}(x) \in (a,b)$. We define
\begin{equation}
h(x):= \zeta^{n} \tilde{h} (f^{-n}(x)).
\end{equation}
Let us first prove that $h$ is well defined. Indeed, assume $x \in (0,1)$ and $m,n \in \N$ are such that $f^{-n}(x) \in (a,b)$ and $f^{-n-m}(x) \in (a,b)$. Writing $y = f^{-n-m}(x)$, we have $f^{m}(y) = f^{n}(x) \in (a,b)$. From~\eqref{eq:iterated_conjugation}, it follows that
\begin{equation}
\zeta^{n} \tilde{h} (f^{-n}(x)) = \zeta^{n} \tilde{h} (f^{m}(y)) = \zeta^{n+m} \tilde{h} (y) = \zeta^{n+m} \tilde{h} (f^{-n-m}(x)),
\end{equation}
which proves that the definition of $h(x)$ does not depend on the choice of $n \in \N$.

Since $\cap_{n \geq 0} f^{n}(a,b) = (0,1)$ and $\zeta>1$, $h$ extends the function $\tilde{h}$ to the whole interval $(0,1)$ and its image is the whole real line. It remains to prove that $h$ is indeed a conjugation of $f$ in the whole interval $(0,1)$. Fix then $x \in (0,1)$ and $n \in \N$ such that $f^{-n} (x) \in (a,b)$. Notice that
\begin{equation}
\zeta h(x) = \zeta^{n+1} \tilde{h} (f^{-n}(x)) = \zeta^{n+1} \tilde{h} (f^{-n-1}(f(x))) = h(f(x)),
\end{equation}
concluding the proof of the claim.

\bibliographystyle{plain}
\bibliography{all}

\end{document}